\documentclass[12pt,twoside]{article}
\usepackage[left=1.5cm,right=1.5cm,top=3cm,bottom=2cm]{geometry}
\renewcommand{\baselinestretch}{1.0}
\usepackage{paralist}
\usepackage{marginnote}
\usepackage{amsfonts}
\usepackage{enumitem}
\usepackage{amsthm}
\usepackage{amsrefs}
\usepackage{amssymb}
\usepackage{amsmath}
\usepackage{graphicx}
\usepackage{varwidth}
\usepackage{hyperref}
\usepackage{fancyhdr}
\usepackage{stmaryrd}
\usepackage{multirow}
\usepackage{blkarray}
\usepackage[autostyle]{csquotes}
\usepackage[mathscr]{euscript}
\hypersetup{
    colorlinks=true,
    linkcolor=black,
    filecolor=magenta,      
    urlcolor=black,
    citecolor=black
}
\DeclareMathAlphabet{\mathcalligra}{T1}{calligra}{c}{h}
\providecommand{\U}[1]{\protect\rule{.1in}{.1in}}
\newtheorem{theorem}{Theorem}
\newtheorem{proposition}[theorem]{Proposition}
\newtheorem{lemma}[theorem]{Lemma}

\newtheorem{remark}[theorem]{Remark}

\DeclareMathOperator{\arccot}{arccot}

\DeclareMathOperator{\inter}{int}

\DeclareMathOperator{\Imag}{Im}
\newcommand{\R}{\mathbb{R}}
\newcommand{\Q}{\mathbb{Q}}

\newcommand{\C}{\mathbb{C}}

\newcommand{\s}{'\emph{ s} }

\newcommand{\nap}{\nabla^{\perp}}

\def\<{{\langle}}
\def\>{{\rangle}}

\def\Jp{J^\perp}
\def\n{\nabla}
\def\d{\partial}
\def\dz{\partial}
\def\dzb{\bar{\partial}}
\def\a{\alpha}

\def\add{\a(\dz,\dz)}

\def\bdd{\b(\dz,\dz)}

\def\id{I}
\def\b{\beta}
\def\g{\gamma}

\def\s{\sigma}

\def\th{\theta}
\def\w{\omega}
\def\bea{\begin{eqnarray*} }
\def\eea{\end{eqnarray*} }
\def\be{\begin{equation} }
\def\ee{\end{equation} }

\def\nap{\nabla^\perp}

\def\qed{\ifhmode\unskip\nobreak\fi\ifmmode\ifinner
\else\hskip5 pt \fi\fi\hbox{\hskip5 pt \vrule width4 pt
height6 pt  depth1.5 pt \hskip 1pt }}

\begin{document}

\title{On the moduli space of isometric surfaces with the same mean curvature in 4-dimensional space forms}
\author{Kleanthis Polymerakis and Theodoros Vlachos}
\date{}
\maketitle

\renewcommand{\thefootnote}{\fnsymbol{footnote}}
\footnotetext{\emph{The first named author has been supported by the Alexander S. Onassis Public Benefit Foundation during his PhD studies
at the University of Ioannina.}}
\renewcommand{\thefootnote}{\arabic{footnote}}

\renewcommand{\thefootnote}{\fnsymbol{footnote}}
\footnotetext{\emph{2010 Mathematics Subject Classification.} 53C42,
53A10.}
\renewcommand{\thefootnote}{\arabic{footnote}}


\begin{abstract}
We study the moduli space of congruence classes of isometric surfaces with the same mean
curvature in 4-dimensional space forms. 
Having the same mean curvature means that there exists
a parallel vector bundle isometry between the normal bundles that preserves the mean curvature vector fields.
We prove that if both Gauss lifts of a compact surface to the twistor bundle
are not vertically harmonic, then there exist at most three nontrivial congruence classes. We show that surfaces with a 
vertically harmonic Gauss lift possess a holomorphic quadratic differential, yielding thus
a Hopf-type theorem. We prove that such surfaces
allow locally a one-parameter family of isometric deformations with the same mean curvature. 
This family is trivial only if the surface is superconformal. For such compact surfaces with non-parallel mean curvature, 
we prove that the moduli space 
is the disjoint union of two sets, each one being either finite, or a circle. In particular, for surfaces in $\R^4$ we prove
that the moduli space is a finite set, under a condition on the Euler numbers of the tangent and normal bundles.
\end{abstract}

\section{Introduction}

A basic problem in surface theory is to understand the role and the importance of the mean curvature.
Bonnet \cite{B} raised the problem to what extent a surface in a complete simply-connected 3-dimensional space form $\Q^3_c$ of curvature $c$,
is determined (up to congruence) by the metric and the mean curvature. 
Generically, a surface in $\Q^3_c$ is uniquely determined by these data. The exceptions are the Bonnet surfaces
that include the constant mean curvature (CMC) surfaces.

There has been a lot of interest in the following natural problem: given an isometric immersion $f\colon M\to \Q^3_c$ of a 2-dimensional 
Riemannian manifold $M$, how many noncongruent isometric immersions of $M$ into $\Q^3_c$ can exist with the same mean curvature with $f$?
This problem has been studied locally or globally by Bonnet \cite{B}, Cartan \cite{Ca}, Lawson \cite{L},
Tribuzy \cite{Tr}, Chern \cite{Ch2}, Roussos-Hernandez \cite{RH} and Kenmotsu \cite{K} among others.
Lawson and Tribuzy \cite{LT} proved that a compact oriented 2-dimensional Riemannian manifold 
admits at most two noncongruent isometric immersions in $\Q_{c}^3$, with the same non-constant mean curvature.
Their result was strengthened recently in \cite{JMN}, under additional assumptions on the isothermicity of the immersion. 
On the other hand,
Lawson \cite{L} proved that if $M$ is simply-connected and $f$ is a CMC surface in $\Q^3_c$,
then the space of isometric immersions with the same mean curvature is the circle $\mathbb{S}^1$, unless $f$ is totally-umbilical.
The case of non-simply-connected CMC surfaces has been studied in \cite{AAMU, BU, ST}.

Surfaces of constant mean curvature have been extensively studied. Hopf \cite{Hopf} showed the existence of a holomorphic
quadratic differential on every CMC surface in $\R^3$, and he proved that a CMC surface of genus zero is a round sphere.
His result was extended to nonflat 3-dimensional space forms by Chern \cite{Chern}. Abresch and Rosenberg \cite{AR} proved that every
CMC surface in the Riemannian products $\mathbb{S}^2\times\R$ and $\mathbb{H}^2\times\R$ possesses a holomorphic quadratic
differential, and they extended Hopf's theorem for such surfaces of genus zero. Their work and an extension of Bonnet's fundamental
theorem (cf. \cite{Da}), led to the study of the Bonnet problem for surfaces in these spaces (cf. \cite{GMM}).
In codimension greater than one, a generalization of 
CMC surfaces are the surfaces whose mean curvature vector field is parallel in the normal connection. 
Existence of holomorphic quadratic differentials, classification results 
and Hopf-type theorems have been proved for parallel mean curvature surfaces in several ambient spaces,
especially in codimension two (cf. \cite{ACT, Fe, LV, FR, TU, Hof, Chen, Yau}).
In particular, in \cite{Ken,KZ} has been proved the existence of parallel mean curvature surfaces in $\mathbb{CH}^2$ that admit non-trivial
isometric deformations preserving the length of the mean curvature vector field.

As a step towards deciphering the role of the mean curvature in codimension two and
inspired by Bonnet's question for surfaces in $\Q^3_c$, we are interested in the following problem:
{\emph{given an isometric immersion $f\colon M\to \Q^4_c$ of a 2-dimensional 
Riemannian manifold $M$, how many noncongruent isometric immersions of $M$ into $\Q^4_c$ can exist with the same mean curvature with $f$?}}
Two isometric immersions $f,\tilde{f} \colon M \to \Q_c^4$ are said to have \emph{the same mean curvature} if there exists
a parallel vector bundle isometry between their normal bundles that preserves the mean curvature vector fields.

The aim of the paper is to study the moduli space $\mathcal{M}(f)$ of congruence classes of isometric immersions that
have the same mean curvature with $f$. For this study, we assign to every surface in $\mathcal{M}(f)$ 
a holomorphic differential with values in the complexified normal bundle of $f$, that measures how far the immersions
deviate from being congruent.
Our results are mostly global in nature. 

We show that the structure of the moduli space of a surface in $\R^4$ is controlled by the behavior of its Gauss map. 
In order to state our first result, we recall that the Grassmannian $Gr(2,4)$ of oriented 2-planes in $\R^4$, can be identified 
with the product $\mathbb{S}^2_{+}\times \mathbb{S}^2_{-}$ of two spheres.
Accordingly, given an isometric immersion $f\colon M\to \R^4$, its Gauss map $g\colon M\to Gr(2,4)$
decomposes into a pair of maps as $g=(g_{+},g_{-})\colon M\to \mathbb{S}^2_{+}\times \mathbb{S}^2_{-}$.
The following result is a Lawson-Tribuzy type theorem \cite{LT} for compact surfaces in $\R^4$.

\begin{theorem} \label{thr4}
Let $f\colon M\to \R^4$ be an isometric immersion of a compact, oriented 2-dimensional 
Riemannian manifold $M$. If both components $g_{+}$ and $g_{-}$
of the Gauss map of $f$ are not harmonic, then there exist at most three
nontrivial congruence classes of isometric immersions of $M$ into $\R^4$, that have the same mean curvature with $f$.
In particular, there exists at most one nontrivial class, if $M$ is homeomorphic to $\mathbb{S}^2$.
\end{theorem}

For surfaces in nonflat space forms, the structure of the moduli space 
is controlled by the behavior of the Gauss lifts $G_{+}\colon M\to {\mathcal Z}_{+}$ and $G_{-}\colon M\to {\mathcal Z}_{-}$.
Here ${\mathcal Z}_{+}$ and ${\mathcal Z}_{-}$ stand for the two connected components of the twistor bundle ${\mathcal Z}$ of $\Q^4_c$.

\begin{theorem}\label{main}
Let $f\colon M\to \Q_{c}^4$ be an isometric immersion of a compact, oriented 2-dimensional Riemannian manifold $M$.
If both Gauss lifts $G_{+}$ and $G_{-}$ of $f$
are not vertically harmonic, then there exist at most three
nontrivial congruence classes of isometric immersions of $M$ into $\Q_{c}^4$, that have the same mean curvature with $f$.
In particular, there exists at most one nontrivial class, if $M$ is homeomorphic to $\mathbb{S}^2$.
\end{theorem}

The above theorem implies that compact surfaces in $\Q^4_c$ whose both Gauss lifts are not vertically harmonic, 
do not allow nontrivial global isometric deformations that preserve the mean curvature. 

It is now interesting to study the non-generic case where at least one of the Gauss lifts is vertically harmonic. 
We recall that Ruh and Vilms \cite{RV}, and later
Jensen and Rigoli \cite{JR}, proved that both Gauss lifts are vertically harmonic if and only if 
the mean curvature vector field is parallel in the normal connection.
Such surfaces are either minimal, or they lie as CMC surfaces in a 
totally geodesic or totally umbilical hypersurface of the ambient space $\Q_c^4$ (cf. \cite{Chen,Yau}). 

Surfaces in $\Q_c^4$ with a vertically harmonic Gauss lift have holomorphic mean curvature vector field and they constitute 
a broader class than parallel mean curvature surfaces. This class contains also non-minimal surfaces with nonflat normal bundle. 
Extensively studied surfaces with a vertically harmonic Gauss lift are the Lagrangian surfaces in $\R^4$ 
with conformal or harmonic Maslov form (cf. \cite{CU1,CU, HR}).
Non-minimal superconformal surfaces in the aforementioned class generalize totally umbilical surfaces and they are characterized by the 
property that their vertically harmonic Gauss lift is holomorphic (we refer to Section \ref{Twist} for details).
Such superconformal surfaces in $\R^4$ have been locally parametrized in terms of minimal surfaces by Dajczer and Tojeiro \cite{DT}.
Here, we show that a non-minimal surface in $\Q_c^4$ with a vertically harmonic Gauss lift
possesses a holomorphic quadratic differential and we prove the following Hopf-type theorem.

\begin{theorem}\label{HT}
Let $f\colon M\to \Q^4_c$ be a non-minimal isometric immersion of a 2-dimensional oriented Riemannian manifold $M$.
If the Gauss lift $G_{\pm}$ of $f$ is vertically harmonic and $M$ is homeomorphic to $\mathbb{S}^2$, 
then $f$ is superconformal. In particular, $f$ is totally umbilical if the Euler number of its normal bundle vanishes.
\end{theorem}

Clearly, the theorem of Hopf-Chern \cite{Hopf,Chern} is an immediate consequence of the above theorem.
This result can be also seen as an extension to the case of non-minimal surfaces, of the well-known theorem of Calabi \cite{Cal}
that a minimal surface of genus zero in the 4-sphere is superminimal.
For surfaces in $\R^4$, an alternative proof was given by Hasegawa \cite{Ha2},
with essential use of the Hyperk\"ahler structure of $\R^4$.

Dajczer and Gromoll \cite{DG2} proved that any simply-connected minimal surface
admits a one-parameter associated family of isometric deformations through minimal surfaces.
This family is trivial if and only if the surface is superconformal.
Extending their result, we are able to produce a new one-parameter family of isometric deformations
that preserve the mean curvature, for any non-minimal surface in $\Q^4_c$ with a vertically harmonic Gauss lift.
It is worth noticing that the second fundamental form of any surface in this family relates to the initial
one in a more involved way than in \cite{DG2}.

\begin{theorem}\label{AF}
Let $f\colon M\to \Q^4_c$ be a non-minimal isometric immersion of a 2-dimensional oriented and simply-connected Riemannian manifold $M$.
If the Gauss lift $G_{\pm}$ of $f$ is vertically harmonic, then:
\begin{enumerate}[topsep=0pt,itemsep=-1pt,partopsep=1ex,parsep=0.5ex,leftmargin=*, label=(\roman*), align=left, labelsep=0em]
\item There exists a one-parameter family of isometric immersions $f^{\pm}_{\th}\colon M\to \Q^4_c$, 
$\theta \in \mathbb{S}^1\simeq \R/2\pi \mathbb{Z}$, which have the same mean curvature with $f^{\pm}_0=f$.
\item If $f$ is superconformal, then $f^{\pm}_{\th}$ is congruent to $f$ for any $\th$.
\item If there exist $\th\neq \tilde{\th} \in \mathbb{S}^1$ such that $f^{\pm}_{\th}$
is congruent to $f^{\pm}_{\tilde{\th}}$, then $f$ is superconformal.
\end{enumerate}
\end{theorem}

The above theorem indicates that non-minimal surfaces with a vertically harmonic Gauss lift inherit some of the properties
of CMC surfaces in 3-dimensional space forms. For instance, we prove in Section 4 that such surfaces satisfy Ricci-like conditions
that extend the Ricci condition for CMC surfaces (cf. \cite{L}). The above family 
can be viewed as an extension of the associated family of CMC surfaces.

Although non-minimal surfaces in $\Q^4_c$ with a vertically harmonic Gauss lift share common properties with both minimal surfaces in $\Q^4_c$ and CMC surfaces
in 3-dimensional space forms, an essential difference between them is that the associated family of the above theorem does not necessarily coincide
with the whole moduli space $\mathcal{M}(f)$. For instance,
if both Gauss lifts are vertically harmonic and $f$ is not totally umbilical, then the moduli space 
is parametrized by $\mathbb{S}^1\times\mathbb{S}^1$.
However, for compact surfaces with a vertically harmonic Gauss lift, using the holomorphicity of the aforementioned quadratic differential, 
the Ricci-like conditions that such surfaces satisfy and estimates for the Euler number of their normal bundle, 
we are able to determine the structure of the moduli space under appropriate geometric or topological assumptions.

\begin{theorem}\label{verha}
Let $f\colon M\to \Q^4_c$ be an isometric immersion of a compact, oriented 2-dimensional Riemannian manifold $M$ with vertically harmonic
Gauss lift $G_{\pm}$.
\begin{enumerate}[topsep=0pt,itemsep=-1pt,partopsep=1ex,parsep=0.5ex,leftmargin=*, label=(\roman*), align=left, labelsep=0em]
\item If $f$ is superconformal, then $\mathcal{M}(f)$ contains at most one nontrivial congruence class.
\item If the mean curvature vector field of $f$ is non-parallel, then the moduli space $\mathcal{M}(f)$
is the disjoint union of two sets, each one being either finite, or the circle $\mathbb{S}^1$.
\item If $c=0$ and the Euler numbers $\chi$ and $\chi_N$ of the tangent and normal bundles satisfy 
$\mathcal{\chi}\neq \mp \mathcal{\chi}_N$, then $\mathcal{M}(f)$ is a finite set.
\end{enumerate}
\end{theorem}

The paper is organized as follows: In Section \ref{s2}, we fix the notation and give some preliminaries. In Section \ref{s3}, we discuss surfaces
with the same mean curvature and assign to a pair of such surfaces the holomorphic differential mentioned earlier.
The moduli space splits into disjoint components.
We study the structure of these components for compact surfaces and then give the proofs of
Theorems \ref{thr4} and \ref{main}. As an application of our results, we provide a short proof of Lawson-Tribuzy theorem and of a recent result \cite{HMW}
for Lagrangian surfaces in $\R^4$. Section \ref{s4} is devoted to surfaces with a vertically harmonic Gauss lift. We prove that such surfaces
satisfy Ricci-like conditions and give the proofs of Theorems \ref{HT} and \ref{AF}. 
As a consequence of Theorem \ref{AF}, we show that the moduli space of
simply-connected surfaces with non-vanishing parallel mean curvature 
vector field is the torus $\mathbb{S}^1\times \mathbb{S}^1$, and the two-parameter associated family coincides
with the one given by Eschenburg-Tribuzy \cite{ET2}. Finally, we give the proof of Theorem \ref{verha}.

\section{Preliminaries}\label{s2}

Throughout the paper, $M$ is a connected, oriented 2-dimensional Riemannian manifold.
Let $f\colon M\to \Q_c^4$ be a surface, i.e., an isometric immersion into the complete 
simply-connected 4-dimensional space form of curvature $c$.
Denote by $N_fM$ the normal bundle of $f$ and by $\nap, R^\perp$ the normal connection and its curvature tensor, respectively.
Let $\a\colon TM\times TM\to N_fM$ be the second fundamental form of $f$ and $A_{\xi}$ the symmetric endomorphism
of $TM$ defined by $\<A_{\xi}X,Y\>=\<\a(X,Y),\xi\>$, where $\xi \in N_fM$ and $\<\cdot,\cdot\>$ stands for the Riemannian metric of $\Q_c^4$.
The Gauss, Codazzi and Ricci equations for $f$ are respectively
\begin{eqnarray} 
&(K-c)\<(X\wedge Y)Z,W\>=\<\a (X,W),\a (Y,Z)\>-\<\a (X,Z),\a (Y,W)\>,&\label{Gauss} \nonumber \\
&(\nap_X \a)(Y,Z)=(\nap_Y \a)(X,Z),&\label{Cod} \nonumber \\
&R^\perp(X,Y)\xi=\a(X,A_{\xi}Y)-\a(A_{\xi}X,Y),&\label{Ric}\nonumber
\end{eqnarray}
where $K$ is the Gaussian curvature, $X,Y,Z,W \in TM$ and $(X\wedge Y)Z=\<Y,Z\>X-\<X,Z\>Y.$

The orientations of $M$ and $\Q^4_c$ induce an orientation on the normal bundle.
The \emph{normal curvature} $K_N$ of $f$ is given by
\be \label{nc}
K_N=\<R^\perp(e_1,e_2)e_4,e_3\>,
\ee
where $\{e_1, e_2\}$ and $\{e_3, e_4\}$ are positively oriented orthonormal frame fields of $TM$ and $N_fM$, respectively.
Notice that if $\tau$ is an orientation-reversing isometry of $\Q_c^4$, then $f$ and $\tau\circ f$ have opposite normal curvatures.
The Gauss and the normal curvatures satisfy the equations
\be \label{normcf}
d\w_{12}=-K \w_1\wedge\w_2,\;\;\;\;d\w_{34}=-K_N \w_1\wedge\w_2,
\ee
where $\{\w_j\}$ is the dual frame field of $\{e_j\}, 1\leq j\leq4$, and the connection forms $\w_{kl},\; 1\leq k,l\leq 4$, are given by
\be \label{connection forms}
d\w_k=\sum_{m=1}^{4}\w_{km}\wedge \w_m,\; 1\leq k \leq 4.
\ee
If $M$ is compact, the Euler-Poincar{\'e} characteristics $\mathcal{\chi}, \mathcal{\chi}_N$ of $TM$ and
$N_fM$, are given respectively, by
\begin{equation*} \label{char}
2\pi\mathcal{\chi}=\int_{M}K,\;\;\;2\pi\mathcal{\chi}_N=\int_{M}K_N.
\end{equation*}

For a symmetric section $\beta \in \Gamma(\text{Hom}(TM\times TM, N_fM))$, the \emph{ellipse associated to} $\beta$ at each $p\in M$
is defined by  $${\cal E}_{\beta}(p)=\left\{\beta(X,X):X\in T_pM, \|X\|=1\right\}.$$
It is indeed an ellipse on $N_{f}M(p)$ centered at $trace\beta(p)/2$, which may degenerate into a line segment or a point.
In particular, the ellipse associated to the second fundamental form is denoted by ${\cal E}_{f}$,
is centered at the mean curvature vector $H$ and is called the
\emph{curvature ellipse of} $f$. It is parametrized by
\be \label{ellipse}
\a(X_{\theta},X_{\theta})=H(p)+\cos 2\theta \ \frac{\a_{11}-\a_{22}}{2}+\sin 2\theta \ \a_{12},
\ee
where $X_{\theta}=\cos \theta e_1 +\sin \theta e_2$, $\a_{ij}=\a(e_i,e_j)$, $i,j=1,2$, and $\{e_1,e_2\}$ is an orthonormal basis of $T_pM$.
The Ricci equation is written equivalently at $p$ as
\be \label{Ricci}
R^\perp(e_1,e_2)=(\a_{11}-\a_{22})\wedge \a_{12}.
\ee
Clearly, the ellipse degenerates into a line segment or a point if and only if the
vectors $(\a_{11}-\a_{22})/2$ and $\a_{12}$ are linearly dependent, or equivalently, if $R^{\perp}=0$ at $p$.
At a point where the curvature ellipse is nondegenerate, $K_N$ is positive if and only if
the orientation induced on the ellipse as $X_{\theta}$ traverses positively the unit tangent 
circle, coincides with the orientation of the normal plane (cf. \cite{GR}). 
Let $\lambda_1,\lambda_2$ be the length of the semiaxes of ${\cal E}_{f}$.
Using the Gauss equation and (\ref{Ricci}), we have that (cf. \cite{Little})
\be \label{axes}
\lambda_1^2+\lambda_2^2=\|H\|^2-(K-c),\;\;\ \lambda_1 \lambda_2=\frac{1}{\pi}A({\cal E}_{f})=\frac{1}{2}|K_N|
\ee
at any point, where $A({\cal E}_{f})$ is the area of the curvature ellipse.
Therefore,
$$\|H\|^2-(K-c)\geq|K_N|.$$
A point $p\in M$ is called \emph{pseudo-umbilic} if the curvature ellipse is a circle at $p$. 
A pseudo-umbilic point is called \emph{umbilic} if the circle degenerates into a point.
From \eqref{axes} it follows that the set $M_0(f)$ of pseudo-umbilic points of $f$ is characterized as
$$M_0(f)=\left\{p\in M: \|H\|^2-(K-c)=|K_N|\right\}.$$
A surface for which any point is pseudo-umbilic is called \emph{superconformal}.
By setting $$M_0^{\pm}(f)=\{p\in M_0(f):\pm K_N\geq0 \},$$ 
it is clear that $M_0(f)=M_0^{+}(f)\cup M_0^{-}(f)$ and the set $M_1(f)$ of umbilic points is
$$M_1(f)=M_0^{+}(f)\cap M_0^{-}(f)=\{p\in M: \|H\|^2=(K-c)\}.$$

For later use we need the following elementary fact.

\begin{lemma}\label{rot}
Let $f\colon M\to \Q^4_c$ be a surface and $\g \in \Gamma(\text{Hom}(TM\times TM, N_fM))$ a symmetric
section. Assume that the ellipse ${\cal{E}}_{\g}$ associated to $\gamma$ is not a circle at a point $p\in M$.
Then, there exist positively oriented orthonormal frame fields $\{e_1,e_2\}$ of $TM$, $\{e_3,e_4\}$ of $N_fM$,
on a neighbourhood $U$ of $p$, and $\kappa,\mu\in \mathcal{C}^{\infty}(U)$ with $\kappa>|\mu|$, such that 
${\g}_{11}-{\g}_{22}=2\kappa e_3$ and ${\g}_{12}=\mu e_4$, where $\g_{ij}=\g(e_i,e_j),\; j=1,2$.
\end{lemma}

\begin{proof}
Let $\{\tilde{e}_1,\tilde{e}_2\}$ be a positively oriented orthonormal tangent frame field around $p$ and set $X_t=\cos te_1+\sin te_2,\; t\in \R$.
The ellipse ${\cal{E}}_{\g}(q)$ is parametrized by
$$\g(X_{t}(q),X_{t}(q))=trace\g(q)/2+\cos 2t u(q) +\sin 2t v(q),$$
where $u=(\tilde{\g}_{11}-\tilde{\g}_{22})/2,\; v=\tilde{\g}_{12}$ and $\tilde{\g}_{ij}=\g(\tilde{e}_i,\tilde{e}_j), i,j=1,2$.
Our assumption implies that at least one of the quantities $\left\|u\right\|-\|v\|$, $\left<u,v\right>$ is non-zero at $p$. 
By continuity, we have that either $\|u\|\neq \|v\|$, or $\left<u,v\right>\neq 0$ everywhere on a neighbourhood $U$ of $p$.
Let $q\in U$. The function $r(t)=\left\|\mathring{\g}(X_{t}(q),X_{t}(q))\right\|^2$, where $\mathring{\g}$ is the traceless part of $\g$,
attains its maximum at $t_0$. Clearly, $\mathring{\g}(X_{t_0}(q),X_{t_0}(q))$ is a major semiaxis of 
${\cal{E}}_{\g}(q)$ and $\mathring{\g}(X_{t_0}(q),X_{t_0+\pi/2}(q))$ is a minor semiaxis.
From $r'(t_0)=0$ and $r''(t_0)\leq0$, we obtain that
$$\sin 4t_0\left(\left\|u\right\|^2-\|v\|^2\right)(q)= 2\cos 4t_0\left<u,v\right>(q)$$
and
$$\cos 4t_0\left(\left\|u\right\|^2-\|v\|^2\right)(q)+2\sin 4t_0\left<u,v\right>(q)\geq0.$$
Define the function $\w\in\mathcal{C}^{\infty}(U)$ by 
$$\w=\frac{1}{4}\arctan\left(\frac{2\left<u,v\right>}{\left\|u\right\|^2-\|v\|^2}\right)\;\; \mbox{modulo}\;\; 2\pi,$$
if $\left\|u\right\|\neq\|v\|$ on $U$,
where the branch of $\arctan$ is such that $\cos 4\w \left(\left\|u\right\|^2-\|v\|^2\right)\geq0$.
If $\left<u,v\right>\neq0$ on $U$, then $\w$ is defined by
$$\w=\frac{1}{4}\arccot\left(\frac{\left\|u\right\|^2-\|v\|^2}{2\left<u,v\right>}\right)\;\; \mbox{modulo}\;\; 2\pi,$$ 
where the branch of $\arccot$ is such that $\sin 4\w \left<u,v\right>\geq0$.
We consider the frame field $e_1=\cos\w \tilde{e}_1+\sin\w \tilde{e}_2,\; e_2=-\sin\w \tilde{e}_1+\cos\w \tilde{e}_2$
and the positively oriented orthonormal frame field $\{e_3,e_4\}$ in the normal bundle such that
$\mathring{\g}(e_1,e_1)=\left\|\mathring{\g}(e_1,e_1)\right\|e_3$.
By the choice of $\w$, we have that $\mathring{\g}(e_1,e_1)$ is a major semiaxis of ${\cal{E}}_{\g}$.
Then, the proof follows with $\kappa=\left\|\mathring{\g}(e_1,e_1)\right\|$ and $\mu=\langle \mathring{\g}(e_1,e_2),e_4\rangle$.
\qed
\end{proof}

\subsection{Complexification and associated differentials}
\
The complexified tangent bundle $TM\otimes\mathbb{C}$ of a 2-dimensional oriented Riemannian manifold $M$, decomposes 
into the eigenspaces of the complex structure $J$, denoted by $T^{(1,0)}M$ and $T^{(0,1)}M$, corresponding to the eigenvalues
$i$ and $-i$, respectively. Let $(U,z=x+iy)$ be 
a local complex coordinate on $M$. The Wirtinger operators are defined on $U$ by $\d=\d_z=(\d_x-i\d_y)/2$, 
$\bar{\d}=\d_{\bar z}=(\d_x+i\d_y)/2$, where $\d_x=\d/\d x$ and $\d_y=\d/\d y$.

Let $E$ be a complex vector bundle over $M$ equipped with a connection $\n^E$. An \emph{$E$-valued differential $\Psi$ of $r$-order}
is an $E$-valued $r$-covariant tensor field on $M$ of holomorphic type $(r,0)$. The $r$-differential $\Psi$
is called \emph{holomorphic} (cf. \cite{BWW}) if its covariant derivative $\n^E \Psi$ has holomorphic type $(r+1,0)$. 
On $U$ with complex coordinate $z$, $\Psi$ has the form $\Psi=\psi dz^r$, where $\psi \colon U\to E$ is given by $\psi=\Psi(\dz,\dots,\dz)$.
Then $\Psi$ is holomorphic if and only if $$\n^E_{\dzb}\psi=0,$$
i.e., $\psi$ is a holomorphic section. The following result, proved in \cite{Ch,BWW}, will be used in the sequel.

\begin{lemma}\label{zeros}
Assume that the $E$-valued differential $\Psi$ is holomorphic and let $p\in M$ be such that $\Psi(p)=0$.
Let $(U,z)$ be a local complex coordinate with $z(p)=0$. Then either $\Psi \equiv 0$ on $U$; or $\Psi=z^m\Psi^{*}$,
where $m$ is a positive integer and $\Psi^{*}(p)\neq0$.
\end{lemma}

Of particular importance for our approach are two quadratic differentials
associated to a surface $f\colon M\to \Q_c^4$, as well as their relation with
the Gauss lifts of $f$ to the twistor bundle.
The second fundamental form can be $\C$-bilinearly extended to $TM\otimes\mathbb{C}$ with values in the complexified
normal bundle $N_fM\otimes\mathbb{C}$ and then decomposed into its $(k,l)$-components $\a^{(k,l)}$, $k+l=2$,
which are tensors of $k$ many 1-forms vanishing on $T^{(0,1)}M$ and $l$ many 1-forms vanishing on $T^{(1,0)}M$.
In terms of a local complex coordinate $z=x+iy$,
the metric $ds^2$ of $M$ is written as $ds^2=\lambda^2|dz|^2$,
where $\lambda > 0$ is the conformal factor.
Setting $e_1=\d_x/\lambda$ and  $e_2=\d_y/\lambda$, the components of $\a$ are given by
$$\a^{(2,0)} = \a(\dz,\dz)dz^2,\;\; \a^{(0,2)} = \overline {\a^{(2,0)}},\;\; \a^{(1,1)} = \a(\dz,\dzb)(dz\otimes d\bar{z} + d\bar{z}\otimes dz),$$
where
\begin{eqnarray} \label{defadd}
\a(\dz,\dz) =\frac{\lambda^2}{2}\big( \frac{\a_{11}-\a_{22}}{2}-i\a_{12}\big),\;\; \a_{ij}=\a(e_i,e_j), i,j=1,2,\;\; \mbox{and}\;\;
\a(\dz,\dzb)= \frac{\lambda^2}{2}H.
\end{eqnarray}
The Codazzi equation is equivalent to
\be \label{CC}
\nap_{\dzb}\add=\frac{\lambda^2}{2}\nap_{\dz}H.
\ee

The {\emph{Hopf differential of $f$}} is the quadratic $N_fM\otimes\mathbb{C}$-valued differential $\Phi=\a^{(2,0)}$ with local expression
$\Phi=\add dz^2$.
It follows from (\ref{CC}) that $\Phi$ is holomorphic if and only if the mean curvature vector field $H$ is parallel. 

Let $\Jp$ be the complex structure of $N_fM$ defined by the metric and the orientation.
The complexified normal bundle decomposes as 
$$N_fM\otimes\mathbb{C}=N_f^{-}M\oplus N_f^{+}M$$ 
into the eigenspaces $N_f^{-}M$ and $N_f^{+}M$ of $\Jp$, corresponding to the eigenvalues $i$ and $-i$, respectively. 
Any section $\xi \in N_fM\otimes\mathbb{C}$ is decomposed as $\xi=\xi^{-}+\xi^{+}$, where
$$\xi^{\pm}=\pi^{\pm}(\xi)=\frac{1}{2}(\xi\pm i\Jp \xi).$$
A section $\xi$ of $N_fM\otimes\mathbb{C}$ is called \emph{isotropic} if at any point of $M$, either
$\xi=\xi^{-}$, or $\xi=\xi^{+}$. This is equivalent to $\langle \xi,\xi\rangle =0$, where
$\langle \cdot,\cdot\rangle$ is the $\C$-bilinear extension of the metric.
Notice that $\langle \zeta, \eta \rangle =0$ for $\zeta \in N^{-}_fM$ and $\eta\in N^{+}_fM$, implies that $\zeta=0$ or $\eta=0$.
According to the above decomposition, the Hopf differential splits as
\begin{equation*}
\Phi=\Phi^{-}+\Phi^{+},\;\;\mbox{where}\;\; \Phi^{\pm}=\pi^{\pm}\circ\Phi.
\end{equation*}

\begin{lemma}\label{pseudo}
\begin{enumerate}[topsep=0pt,itemsep=-1pt,partopsep=1ex,parsep=0.5ex,leftmargin=*, label=(\roman*), align=left, labelsep=-0.4em]
\item The zero-sets of $\Phi^{\pm}$ and $\Phi$, are $M_{0}^{\pm}(f)$ and $M_1(f)$, respectively.
\item The surface $f$ is superconformal with normal curvature $\pm K_N\geq0$ if and only if $\Phi^{\pm}\equiv0$. 
In particular, if $f$ is superconformal, then $K_N$ vanishes precisely on $M_1(f)$.
\end{enumerate}
\end{lemma}

\begin{proof}
In terms of a local complex coordinate $z$ around a point $p$, it is clear that
$\Phi^{\pm}(p)=0$ if and only if $\Jp \add=\pm i\add$ at $p$.
It follows from (\ref{defadd}) that $\mathcal{E}_f(p)$ is a circle if and only if $\add$ is isotropic.
Bearing in mind \eqref{ellipse} and \eqref{Ricci}, we conclude that $\Phi^{\pm}(p)=0$ if and only
if $\mathcal{E}_f(p)$ is a circle and $\pm K_N(p)\geq0$.
Obviously, $\Phi$ vanishes precisely at the points where both $\Phi^{\pm}$ vanish, i.e., the umbilic points. 
This proves part (i), and the first assertion of part (ii) follows immediately.
If $f$ is superconformal, then the second equation in \eqref{axes} implies that the normal curvature vanishes 
precisely at the umbilic points.
\qed
\end{proof}

\subsection{Twistor spaces and Gauss lifts} \label{Twist}
\
We recall some known facts about the twistor theory of 4-dimensional space forms.
The reader may consult \cite{ES, Fr}, although the paper of Jensen and Rigoli \cite{JR} is closer to our approach.
Let $O(\Q_c^4)$ be the principal $O(4)$-bundle of orthonormal frames in $\Q_c^4$, which has two connected
components denoted by $O_{+}(\Q_c^4)$ and $O_{-}(\Q_c^4)$, corresponding to the two connected components of $O(4)$.
The twistor bundle ${\mathcal Z}$ of $\Q_c^4$ is defined as the set of all pairs $(p,\tilde{J})$, where $p\in \Q_c^4$
and $\tilde{J}$ is an orthogonal complex structure on $T_p\Q_c^4$. The twistor projection $\varrho \colon {\mathcal Z}\to \Q_c^4$
is defined by $\varrho (p,\tilde{J})=p$, and ${\mathcal Z}$ is an $O(4)/U(2)$-fiber bundle over $\Q_c^4$, which is associated to
$O(\Q_c^4)$. Indeed, at a point $p\in \Q_c^4$ and for any orthonormal frame $e=(e_1,e_2,e_3,e_4)$ of $T_p\Q_c^4$,
define an orthogonal complex structure $\tilde{J}_e$ by $$\tilde{J}_e e_1=e_2,\ \tilde{J}_e e_3=e_4,\ \tilde{J}_e^2=-\id.$$
Any orthogonal complex structure on $T_p\Q_c^4$ is equal to $\tilde{J}_e$ for some orthonormal frame $e$ of $T_p\Q_c^4$ and 
$\tilde{J}_e=\tilde{J}_{\tilde{e}}$ if and only if $\tilde{e}=eA$ for some $A\in U(2)$. Thus, the set of all
orthogonal complex structures on $T_p\Q_c^4$ is $O(4)/U(2)$ and has two connected components isomorphic to
$SO(4)/U(2)=\{\tilde{J}_e: e\ \mbox{is a}\pm \mbox{oriented frame of}\ T_p\Q_c^4\}$. Hence, the twistor bundle is
$${\mathcal Z}=O(\Q_c^4)\times_{O(4)}O(4)/U(2)=O(\Q_c^4)/U(2)$$
and its two connected components are denoted by ${\mathcal Z}_{+}$ and ${\mathcal Z}_{-}$. 
Each projection $\varrho_{\pm}\colon {\mathcal Z}_{\pm}\to \Q_c^4$ is a $P^1(\C)\simeq \mathbb{S}^2$-fiber bundle over $\Q_c^4$. 

A one-parameter family of Riemannian metrics $g_{t},\ t>0$, is defined on ${\mathcal Z}$ in a natural way, making 
$\varrho_{+}$ and $\varrho_{-}$ Riemannian submersions. 
With respect to the (common) decomposition of the tangent bundle of ${\mathcal Z}_{\pm}$ induced by the Levi-Civit\'{a}
connection of $g_{t}$ 
$$T{\mathcal Z}_{\pm}=T^h{\mathcal Z}_{\pm}\oplus T^v{\mathcal Z}_{\pm}$$
into horizontal and vertical subbundles, the metric $g_{t}$ is given by the pull-back of the metric of $\Q_c^4$ to the
horizontal subspaces and by adding the $t^2$-fold of the metric of the fibers.

Denote by $Gr_2(T\Q_c^4)$ the Grassmann bundle of oriented 2-planes tangent to $\Q_c^4$. There are projections
$$\Pi_{+}\colon Gr_2(T\Q_c^4)\to {\mathcal Z}_{+}\;\;\; \mbox{and} \;\;\; \Pi_{-}\colon Gr_2(T\Q_c^4)\to {\mathcal Z}_{-}$$
defined as follows; if $\zeta \subset T_p\Q_c^4$ is an oriented
2-plane, then $\Pi_{\pm}(p,\zeta)$ is the complex structure on $T_p\Q_c^4$ corresponding to the rotation by
$+\pi/2$ on $\zeta$ and the rotation by $\pm \pi/2$ on $\zeta^{\perp}$.
The Gauss lift $G_f\colon M\to Gr_2(T\Q_c^4)$, of an oriented surface $f\colon M\to \Q^4_c$ is defined by $G_f(p)=(f(p),f_{*}T_pM)$.
The \emph{Gauss lifts of $f$ to the twistor bundle} are the maps
$$G_{+}\colon M\to {\mathcal Z}_{+}\;\; \mbox{and}\;\;G_{-}\colon M\to {\mathcal Z}_{-},\;\;\mbox{where}\;\;  G_{\pm}=\Pi_{\pm}\circ G_f.$$
At any point $p\in M$, we obviously have $G_{\pm}(p)=(f(p),\tilde{J}_{\pm}(f(p)))$,
where 
\begin{equation*}
\tilde{J}_{\pm}(f(p))= \left\{
\begin{array}{rll}
f_{*}\circ J(p), & \mbox{on}\ f_{*}T_pM,\\
\pm \Jp(p), & \mbox{on}\ N_fM(p).
\end{array}\right.
\end{equation*}
The Gauss lift $G_{\pm}\colon M\to ({\mathcal Z}_{\pm},g_t)$ is called \emph{vertically harmonic} if its tension field has
vanishing vertical component with respect to the decomposition $T{\mathcal Z}_{\pm}=T^h{\mathcal Z}_{\pm}\oplus T^v{\mathcal Z}_{\pm}$.

Let $\{e_j\},1\leq j\leq4,$ be a $\pm$ oriented, local adapted orthonormal frame field of $\Q^4_c$, where $\{e_1,e_2\}$ is in the orientation of $TM$.
Denote by $\{\w_j\},1\leq j\leq4,$ the corresponding coframe and by $\w_{kl},\; 1\leq k,l\leq 4$, the connection forms 
given by (\ref{connection forms}).
The pull-back of $g_t$ on $M$ under $G_{\pm}$, is related to the metric $ds^2$ of $M$ as follows
$$G_{\pm}^{*}(g_t)=ds^2 + \frac{t^2}{4}\left((\w_{13}-\w_{24})^2 + (\w_{14}-\w_{23})^2\right).$$
The covariant differential of the mean curvature vector field $H=H^3e_3+H^4e_4$ is given by
\begin{eqnarray}
\nap H &=& \sum_{a=3}^{4}\big(dH^a + \sum_{b=3}^{4}H^b \w_{ba}\big)\otimes e_a = \sum_{j=1}^{2}\sum_{a=3}^{4}H^a_j \w_{j}\otimes e_a. \label{Hij}
\end{eqnarray}
The following proposition relates the vertical harmonicity of the Gauss lift $G_{\pm}$ with the holomorphicity of 
the differential $\Phi^{\pm}$ and the holomorphicity of the section $H^{\pm}$.
The equivalence of (i) and (iv) below, is a slight modification of Theorem 8.1. in \cite{JR} for space forms,
where the scalar curvature of $\Q^4_c$ is normalized to be equal to $c$. It was also proved
by Hasegawa \cite{Ha} who studied surfaces with a vertically harmonic Gauss lift.

\begin{proposition}\label{glphi}
Let $f\colon M \to \Q^4_c$ be a surface with mean curvature vector field $H$. The following are equivalent:
\begin{enumerate}[topsep=0pt,itemsep=-1pt,partopsep=1ex,parsep=0.5ex,leftmargin=*, label=(\roman*), align=left, labelsep=0em]
\item The Gauss lift $G_{\pm}\colon M\to ({\mathcal Z}_{\pm},g_t)$ of $f$ is vertically harmonic.
\item The differential $\Phi^{\pm}$ is holomorphic.
\item The section $H^{\pm}$ is anti-holomorphic.
\item $\nap_{JX}H=\pm\Jp \nap_{X}H$, for any $X\in TM$.
\end{enumerate}
\end{proposition}

\begin{proof}
The equivalence of (ii), (iii) and (iv) is an immediate consequence of the Codazzi equation (\ref{CC}).
We prove that (i) is equivalent to (iv).
The tension field of $G_{\pm}$, in terms of an appropriate frame field $\{E^{\pm}_k,\; 1\leq k \leq 6\}$ of ${\mathcal Z}_{\pm}$, 
is given by (cf. \cite{JR})
$$\tau(G_{\pm})=\sum_{k=1}^{6}B^{\pm}_k E^{\pm}_k,$$
where
\begin{eqnarray}
B^{\pm}_j &=& 0\; \mbox{for}\; j=1,2; \;\; B^{\pm}_a = 2H^a(1-ct^2) \; \mbox{for} \; a=3,4, \nonumber \\
B^{\pm}_5 &=& 2t(H^4_2-H^3_1),\;\; B^{\pm}_6 = -2t(H^4_1+H^3_2). \nonumber
\end{eqnarray}
Its vertical component is given by $$(\tau(G_{\pm}))^v=B^{\pm}_5 E^{\pm}_5 + B^{\pm}_6 E^{\pm}_6$$
and therefore, $G_{\pm}$ is vertically harmonic if and only if $H^4_2=H^3_1$ and $H^4_1=-H^3_2$.
Using (\ref{Hij}) and taking into account the orientation of the frame field of $\Q^4_c$, the result follows after a straightforward computation.
\qed
\end{proof} 
\medskip

From the above proof it follows that if $t^2=1/c$, then $G_{\pm}$ is vertically harmonic if and only if it is harmonic. 

It is clear from Proposition \ref{glphi} that both Gauss lifts are vertically harmonic if and only if the surface has
parallel mean curvature vector field in the normal connection.

Proposition \ref{glphi} and Lemma \ref{pseudo}(ii) imply that any superconformal surface $f\colon M\to \Q^4_c$
with $\pm K_N\geq0$ has vertically harmonic Gauss lift $G_{\pm}$. The Gauss lift $G_{\pm}$ of
such surfaces is holomorphic with respect to a complex structure $\mathcal{J}$ on $\mathcal{Z}$, that
makes $(\mathcal{Z},g_t)$ a Hermitian manifold (cf. \cite{ES, JR}). The following proposition
shows that the converse is also true for non-minimal superconformal surfaces.

\begin{proposition} \label{HOLGL}
Let $f\colon M\to \Q^4_c$ be a non-minimal superconformal surface. If the Gauss lift $G_{\pm}$ of $f$ is vertically
harmonic, then $\Phi^{\pm}\equiv0$.
\end{proposition}

\begin{proof}
Arguing indirectly, assume that $\Phi^{\pm}\not\equiv0$. 
From Proposition \ref{glphi}, we know that $\Phi^{\pm}$ is holomorphic
and Lemma \ref{zeros} implies that its zeros are isolated.
From Lemma \ref{pseudo}(ii) it follows that $\Phi^{\mp}\equiv0$ and consequently
$\Phi$ is holomorphic. Then, the mean curvature vector field of $f$ is parallel. 
Hence, $K_N=0$ on $M$ and Lemma \ref{pseudo}(ii) implies that $f$ is totally umbilical, a contradiction.
\qed
\end{proof}
\medskip

\begin{remark} \label{Eucl}
{\em In the case of $\R^4$, $({\mathcal Z}_{\pm},g_t)$ is isometric to the product $\R^4 \times \mathbb{S}^2(t)$. The Grassmann bundle is trivial
$Gr_2(\R^4)\simeq \R^4 \times Gr(2,4)$ and the Gauss lift of $f$ to the Grassmann bundle is given by $G_f=(f,g)$, where
$g=(g_{+},g_{-})\colon M \to \mathbb{S}^2_{+} \times \mathbb{S}^2_{-}$ is the Gauss map of $f$ and the radius of 
$\mathbb{S}^2_{\pm}$ is $1/\sqrt2$.
The Gauss lift $G_{\pm}$ of $f$ to the twistor bundle is then given by $G_{\pm}=(f,\sqrt2tg_{\pm})$ and it is vertically harmonic if and only if
$g_{\pm}$ is harmonic.}
\end{remark}

\section{Surfaces with the same mean curvature}\label{s3}

\subsection{Bonnet pairs and the distortion differential}

Let $M$ be a 2-dimensional oriented Riemannian manifold and $f,\tilde{f}\colon M\to \Q^4_c$ be isometric immersions
with second fundamental forms $\a, \tilde{\a}$ and mean curvature vector fields $H, \tilde{H}$, respectively. 
The surfaces $f,\tilde{f}$ are said to have {\emph{the same 
mean curvature}}, if there exists a parallel vector bundle isometry $T\colon N_{f}M\to N_{\tilde f}M$
such that $TH=\tilde{H}$.

The case of minimal surfaces has been studied in \cite{DG1} and \cite{Vl2}.
In this section, we assume that all surfaces under consideration are non-minimal.

Suppose that $f,\tilde{f}\colon M \to \Q^4_c$ have the same mean curvature and let $T\colon N_{f}M\to N_{\tilde{f}}M$
be a parallel vector bundle isometry  satisfying $TH=\tilde{H}$.
After an eventual composition of one of the surfaces with an
orientation-reversing isometry of $\Q^4_c$, we may hereafter suppose that $T$ is orientation-preserving.
To such a pair $(f,\tilde{f})$ we assign a holomorphic differential which is going to play a fundamental role in the sequel.
The section of $\text{Hom}(TM\times TM,N_{f}M)$ given by 
$$D^T_{f,\tilde{f}}= \a-T^{-1}\circ \tilde{\a}$$ 
measures how far the surfaces deviate from being congruent.
Since $D^T_{f,\tilde{f}}$ is traceless, its $\mathbb{C}$-bilinear extension decomposes into its $(k,l)$-components, $k+l=2$, as
$$D^T_{f,\tilde{f}}=(D^T_{f,\tilde{f}})^{(2,0)}+(D^T_{f,\tilde{f}})^{(0,2)},\;\; 
\mbox{where}\;\; (D^T_{f,\tilde{f}})^{(0,2)}=\overline{(D^T_{f,\tilde{f}})^{(2,0)}}.$$
We are interested into the $(2,0)$-part which is given by
$$Q^T_{f,\tilde{f}}=(D^T_{f,\tilde{f}})^{(2,0)}=\Phi-T^{-1}\circ \tilde{\Phi},$$
where $\Phi, \tilde{\Phi}$ stand for the Hopf differentials of $f,\tilde{f}$, respectively.

\begin{lemma}\label{qke}
Let $f,\tilde{f}\colon M\to \Q^4_c$ be surfaces and $T\colon N_{f}M\to N_{\tilde{f}}M$ an
orientation-preserving, parallel vector bundle isometry satisfying $TH=\tilde{H}$. Then:
\begin{enumerate}[topsep=0pt,itemsep=-1pt,partopsep=1ex,parsep=0.5ex,leftmargin=*, label=(\roman*), align=left, labelsep=-0.4em]
\item The quadratic differential $Q^T_{f,\tilde{f}}$ is holomorphic and independent of $T$. 
\item The normal curvatures of the surfaces are equal and
the curvature ellipses $\mathcal{E}_{f}$, $\mathcal{E}_{\tilde{f}}$ are congruent at any point of $M$.
In particular, $M^{\pm}_0(f)=M^{\pm}_0(\tilde{f})$.
\end{enumerate}
\end{lemma}

\begin{proof}
(i) From our assumption it follows that the section $T^{-1}\circ \tilde{\a}$
of $\text{Hom}(TM\times TM, N_fM)$ satisfies the Codazzi equation
for the data on $N_{f}M$ and thus, $Q^T_{f,\tilde{f}}$ is holomorphic by (\ref{CC}). 

Suppose that there exists another orientation-preserving
parallel vector bundle isometry $S\colon N_{f}M\to N_{\tilde{f}}M$ with $SH=\tilde{H}$.
We argue that $Q^T_{f,\tilde{f}}\equiv Q^S_{f,\tilde{f}}$.
Set $L=T^{-1}\circ S$
and $U=\{p\in M: H(p)\neq0\}$. On $N_{f}U$, $L$ preserves both of $H$ and $\Jp H$ and thus, $T=S$ on $N_{f}U$. 
Therefore, the holomorphic differential $Q^T_{f,\tilde{f}}-Q^S_{f,\tilde{f}}$ vanishes identically on the open subset $U$ of $M$. 
Then by Lemma \ref{zeros}, we obtain that $Q^T_{f,\tilde{f}}\equiv Q^S_{f,\tilde{f}}$ on $M$.

(ii) The vector bundle isometry $T$ preserves the normal curvature tensors. Since it is orientation-preserving, 
(\ref{nc}) implies that the normal curvatures of $f,\tilde{f}$ are equal.
The fact that the curvature ellipses are congruent, now follows from
(\ref{axes}) and this completes the proof.
\qed
\end{proof}
\medskip

Lemma \ref{qke}(i) allows us to assign to each pair of surfaces $(f,\tilde{f})$ with the same mean curvature,
a holomorphic differential denoted by $Q_{f,\tilde{f}}$, which is called \emph{the distortion differential of the pair} and is given by
$$Q_{f,\tilde{f}}=\Phi-T^{-1}\circ \tilde{\Phi}.$$
Obviously, $Q_{f,\tilde{f}}\equiv 0$ if and only if $f$ and $\tilde{f}$ are congruent.
To simplify the notation, we denote the distortion differential associated to the pair
$(f,\tilde{f})$ by $Q$, whenever there is no danger of confusion.
A pair $(f,\tilde{f})$ of noncongruent surfaces with the same mean curvature is called a \emph{Bonnet pair}.
In this case, the zero-set of $Q$ is denoted by $Z$ and according to Lemmas \ref{zeros} and \ref{qke}(i), consists of isolated points only.

With respect to the decomposition $N_fM\otimes\mathbb{C}=N_f^{-}M\oplus N_f^{+}M$, the distortion differential $Q$ splits as 
\begin{eqnarray}
Q=Q^{-}+Q^{+},\;\; \mbox{where }\;\;Q^{\pm}=\pi^{\pm}\circ Q.\nonumber
\end{eqnarray}
It follows from Lemma \ref{qke}(i) that each differential
\be\label{qpr}
Q^{\pm}=\Phi^{\pm}-T^{-1}\circ \tilde{\Phi}^{\pm}
\ee
is holomorphic.
According to Lemma \ref{zeros}, either $Q^{\pm}\equiv 0$, or its zero-set $Z^{\pm}$ consists of isolated points only.

\subsection{The decomposition of the moduli space}

Let $f\colon M\to \Q^4_c$ be a non-minimal oriented surface.
We denote by ${\mathcal M}(f)$ {\emph {the moduli space of congruence classes of all isometric immersions 
of $M$ into $\Q^4_c$, that have the same mean curvature with $f$}}.
Since the distortion differential of a Bonnet pair does not vanish identically, the moduli space
can be written as 
$$\mathcal{M}(f)=\mathcal{N}^{-}(f)\cup \mathcal{N}^{+}(f)\cup\{f\},$$
where
$$\mathcal{N}^{\pm}(f)=\{\tilde{f}:\; Q^{\pm}_{f, \tilde{f}}\not \equiv 0\}/\text{Isom}^+(\Q^4_c),$$
$\{f\}$ is the trivial congruence class and $\text{Isom}^+(\Q^4_c)$ is the group of orientation-preserving isometries of $\Q^4_c$.
Furthermore, $\mathcal{M}(f)$ decomposes into disjoint components as
$$\mathcal{M}(f)= \mathcal{M}^*(f)\cup \mathcal{M}^{-}(f)\cup \mathcal{M}^{+}(f)\cup\{f\},$$ 
where 
$$\mathcal{M}^{\pm}(f)=\mathcal{N}^{\pm}(f)\smallsetminus \mathcal{N}^{\mp}(f)= \{\tilde{f}:\; Q_{f, \tilde{f}}\equiv Q^{\pm}_{f, \tilde{f}}\}/\text{Isom}^+(\Q^4_c),$$
and
$$\mathcal{M}^*(f)= \mathcal{N}^{-}(f)\cap \mathcal{N}^{+}(f)=
\{\tilde{f}:\; Q^{-}_{f, \tilde{f}}\not\equiv 0\;\; \mbox{and}\;\; Q^{+}_{f, \tilde{f}}\not\equiv 0\}/\text{Isom}^+(\Q^4_c).$$
Hereafter, whenever we refer to a surface in the moduli space we mean its congruence class.

The following theorem provides information about the structure of the components of the decomposition of the moduli space
${\mathcal M}(f)$ of a compact surface and is fundamental for the proofs of our results.

\begin{theorem} \label{basic}
Let $f\colon M\to \Q^4_c$ be an isometric immersion of a compact, oriented 2-dimensional Riemannian manifold.
\begin{enumerate}[topsep=0pt,itemsep=-1pt,partopsep=1ex,parsep=0.5ex,leftmargin=*, label=(\roman*), align=left, labelsep=-0.4em]
\item If the Gauss lift $G_{\pm}$ of $f$ is not vertically harmonic, then $\mathcal{M}^{\pm}(f)$ contains at most one congruence class.
Moreover, if $\tilde{f}\in \mathcal{M}^*(f)$ then $\mathcal{M}^*(f)\cup\mathcal{M}^{\pm}(f)=\{\tilde{f}\}\cup\mathcal{M}^{\mp}(\tilde{f})$.
\item If both Gauss lifts of $f$ are not vertically harmonic, then $\mathcal{M}^{*}(f)$ contains at most one congruence class.
In particular, $\mathcal{M}^{*}(f)=\emptyset$ if $M$ is homeomorphic to $\mathbb{S}^2$.
\end{enumerate}
\end{theorem}

For the proof of the above theorem we need a series of auxiliary results.
In view of Lemma \ref{qke}(ii), we denote by $M_0=M_0^{-}\cup M_0^{+}$ and $M_1$ 
the set of pseudo-umbilic and umbilic points of a Bonnet pair, respectively.

\begin{lemma} \label{isolated}
For any Bonnet pair $(f,\tilde{f})$ we have that $M_1$ is isolated and $M^{\pm}_0\subset Z^{\pm}$.
In particular, $M_0^{\pm}$ is isolated if $\tilde{f}\in \mathcal{N}^{\pm}(f)$.
\end{lemma}

\begin{proof}
The fact that $M^{\pm}_0\subset Z^{\pm}$ follows immediately from Lemma \ref{pseudo}(i) and (\ref{qpr}).
Since $M_1\subset Z$ and $Z$ consists of isolated points, the umbilic points are isolated.
If $\tilde{f}\in \mathcal{N}^{\pm}(f)$, then $Z^{\pm}$ is isolated and this completes the proof.
\qed
\end{proof}

\begin{proposition}\label{angles}
If $\tilde{f}\in \mathcal{N}^{\pm}(f)$, then there exists
$\th^{\pm} \in \mathcal{C}^{\infty}(M\smallsetminus Z^{\pm})$
with values in $(0,2\pi)$, such that the distortion differential of the pair $(f,\tilde{f})$ satisfies
\begin{eqnarray} 
Q^{\pm}&=&(1-e^{\mp i\th^{\pm}})\Phi^{\pm}\;\; \mbox{on}\;\; M\smallsetminus Z^{\pm}. \label{general1}
\end{eqnarray}
\end{proposition}

\begin{proof}
We set $\beta=T^{-1}\circ \tilde{\a}$, where $T\colon N_fM\to N_{\tilde f}M$ is an orientation and mean curvature vector field-preserving,
parallel vector bundle isometry.

If $\inter (M_0)\neq \emptyset$, then from Lemma \ref{isolated} we obtain that $\inter (M_0)\subset M^{\mp}_0$. 
Lemma \ref{pseudo}(ii) implies that $\pm K_N<0$, $\Phi^{\mp}\equiv0$ and $\Phi^{\pm}\neq0$ on $\inter (M_0\smallsetminus Z^{\pm})$.
Let $z$ be a local complex coordinate defined on a simply-connected neighbourhood $V\subset\inter (M_0\smallsetminus Z^{\pm})$.
From Lemma \ref{qke}(ii), it follows that the isotropic sections $\add$ and $\bdd$ have the same length.
Hence, there exists $\tau\in \mathcal{C}^{\infty}(V)$ with values in  $(0,2\pi)$, such that
\begin{equation*}
\bdd=J_{\tau}^{\perp}\add,
\end{equation*}
where the rotation $J_{\tau}^{\perp}=\cos\tau I +\sin\tau \Jp$ satisfies $J_{\tau}^{\perp}=e^{\mp i\tau}I$ on $N^{\pm}_{f}M$.
Since $\Phi^{\pm}\neq 0$ on $\inter (M_0\smallsetminus Z^{\pm})$, the function $\tau$ is well-defined modulo $2\pi$
on $\inter (M_0\smallsetminus Z^{\pm})$. 
Moreover, it is non-vanishing modulo $2\pi$ on $\inter (M_0\smallsetminus Z^{\pm})$ and thus, 
there exists a branch in $\mathcal{C}^{\infty}(\inter (M_0\smallsetminus Z^{\pm}))$ with values in $(0,2\pi)$.
By setting $\th^{\pm}=\tau$, we have that \eqref{general1} holds on $\inter (M_0\smallsetminus Z^{\pm})$.
In particular, the assertion is obvious if $M=M_0$.

Assume that $M\neq M_0$ and let $p\in M\smallsetminus M_0$. According to Lemma \ref{rot}, there exist
smooth frame fields $\{e_1,e_2,e_3,e_4\}, \{\tilde{e}_1,\tilde{e}_2,\tilde{e}_3,\tilde{e}_4\}$
on a neighbourhood $U\subset M\smallsetminus M_0$ of $p$, such that 
$${\a}_{11}-{\a}_{22}=2\kappa e_3,\;\; {\a}_{12}=\mu e_4,\;\; \mbox{where}\;\; \a_{ij}=\a(e_i,e_j),\; j=1,2,$$
and 
$${\b}_{11}-{\b}_{22}=2\tilde{\kappa}\tilde{e}_3,\;\; {\b}_{12}=\tilde{\mu}\tilde{e}_4,\;\; \mbox{where}\;\; \b_{ij}=\b(\tilde{e}_i,\tilde{e}_j),\; j=1,2.$$
Lemma \ref{qke}(ii) yields that the ellipses ${\cal{E}}_{f}(q)$ and ${\cal{E}}_{\b}(q)$ are congruent at any point $q\in U$ 
and consequently, $\kappa=\tilde{\kappa}$. 
Using \eqref{nc} and \eqref{Ricci}, we obtain that $K_N=2\kappa \mu$ and ${\tilde K}_N=2\tilde{\kappa}\tilde{\mu}$.
Then, Lemma \ref{qke}(ii) implies that $\mu=\tilde{\mu}$.
Setting $\tilde{e}_3-i\tilde{e}_4=e^{i\th}(e_3-ie_4)$ for some $\th\in \mathcal{C}^{\infty}(U)$, we have that
$$J_{\th}^{\perp}\left(\a_{11}-\a_{22} \right)=\b_{11}-\b_{22}\;\; \mbox{and}\;\; J_{\th}^{\perp}\a_{12}=\b_{12}\;\; \mbox{on}\;\; U,$$
where $J_{\th}^{\perp}=\cos\th I +\sin\th \Jp$. This gives
\begin{equation*}
\b(\tilde{e}_1-i\tilde{e}_2,\tilde{e}_1-i\tilde{e}_2)=J_{\th}^{\perp}\left(\a(e_1-ie_2,e_1-ie_2) \right).
\end{equation*}
Setting $\tilde{e}_1-i\tilde{e}_2=e^{i\s}(e_1-ie_2)$ for some $\s\in \mathcal{C}^{\infty}(U)$,
the above is written equivalently as
\begin{equation*}
T^{-1}\circ \tilde{\Phi}=e^{i\th^{-}}\Phi^{-}+e^{-i\th^{+}}\Phi^{+},\;\; \mbox{where}\;\; \th^{\pm}=\th\pm2\s.
\end{equation*}
Since $\Phi^{-}$ and $\Phi^{+}$ are everywhere non-vanishing on $M\smallsetminus M_0$, 
the functions $\th^{-}$ and $\th^{+}$ are well-defined modulo $2\pi$ on $M\smallsetminus M_0$.
From the assumption $\tilde{f}\in \mathcal{N}^{\pm}(f)$, it follows that $\th^{\pm}$ is non-vanishing modulo $2\pi$
on $M\smallsetminus(M_0\cup Z^{\pm})$ and thus, there exists a branch in $\mathcal{C}^{\infty}(M\smallsetminus(M_0\cup Z^{\pm}))$
with values in $(0,2\pi)$. Obviously, (\ref{general1}) holds on $M\smallsetminus(M_0\cup Z^{\pm})$.

Lemma \ref{qke}(ii) implies that for a point $q\in M_0\smallsetminus (\inter (M_0)\cup Z^{\pm})$,
there exists a unique number $l(q)\in (0,2\pi)$ such that 
$$T^{-1}\circ \tilde{\Phi}(q)=\Jp_{l(q)}\Phi(q),$$
where the rotation is given by $\Jp_{l(q)}=e^{\mp il(q)}I$, since $q\in M^{\mp}_0$.
We extend $\th^{\pm}$ on $M\smallsetminus Z^{\pm}$
by setting $\th^{\pm}(q)=l(q)$.
Then, (\ref{general1}) holds on $M\smallsetminus Z^{\pm}$.
Since $Q^{\pm}$ and $\Phi^{\pm}$ are everywhere non-vanishing on 
$M\smallsetminus Z^{\pm}$, from (\ref{general1}) it follows that $\th^{\pm}$ is smooth.
\qed
\end{proof}
\medskip

Let $f\colon M\to \Q^4_c$ be a surface with Hopf differential $\Phi$.
In terms of a complex chart $(U,z)$, $\Phi^{\pm}$ is written as
\be \label{sf}
\Phi^{\pm}=\phi^{\pm}dz^2.
\ee
Moreover, there exist smooth complex functions $h^{+}$ and $h^-$ such that
\begin{eqnarray} \label{defh}
\nap_{\dzb}\phi^{\pm}&=&h^{\pm}\phi^{\pm}\;\;\;\; \mbox{on}\;\;\;\; U\smallsetminus M_0^{\pm}.
\end{eqnarray}
The following lemma is essential for the proof of Theorem \ref{basic}.

\begin{lemma}
Let $\tilde{f}\in \mathcal{N}^{\pm}(f)$. In terms of a complex chart $(U,z)$,
the function $\th^{\pm}$ in Proposition \ref{angles} satisfies on $U\smallsetminus Z^{\pm}$ the equations
\begin{eqnarray}
&A^{\pm}e^{\pm2i\th^{\pm}}-2i(\Imag A^{\pm})e^{\pm i\th^{\pm}}-{\overline{A^{\pm}}}=0,&\label{B}\\
&\th^{\pm}_{z \overline{z}}= \mp A^{\pm}(1-e^{\pm i\th^{\pm}}),&\label{lm}
\end{eqnarray}
where $A^{\pm}=i\left(h^{\pm}_z-|h^{\pm}|^2\right)$.
\end{lemma}

\begin{proof}
By differentiating (\ref{general1}) with respect to $\bar \d$ in the normal connection, and using (\ref{defh}) and the holomorphicity of $Q^{\pm}$, we obtain
\begin{eqnarray*}
\left(h^{\pm}(1-e^{\mp i\th^{\pm}})\pm ie^{\mp i\th^{\pm}}\th^{\pm}_{\bar z}\right)\phi^{\pm}=0.
\end{eqnarray*}
Since $\phi^{\pm}\neq0$ on $U\smallsetminus Z^{\pm}$, we have
$$\th^{\pm}_{\bar z}=\mp ih^{\pm}(1-e^{\pm i\th^{\pm}})\;\;\; \mbox{and}\;\;\; \th^{\pm}_{z}=\pm i\overline{h^{\pm}}(1-e^{\mp i\th^{\pm}}).$$
Differentiating the above, we find
\begin{equation*}
\th^{\pm}_{\overline{z}z}=\mp A^{\pm}(1-e^{\pm i\th^{\pm}})\;\;\; \mbox{and}\;\;\; 
\th^{\pm}_{z\overline{z}}=\mp \overline{A^{\pm}}(1-e^{\mp i\th^{\pm}}),
\end{equation*}
from which (\ref{B}) and \eqref{lm} follow immediately.
\qed
\end{proof}

\begin{lemma} \label{tri}
\begin{enumerate}[topsep=0pt,itemsep=-1pt,partopsep=1ex,parsep=0.5ex,leftmargin=*, label=(\roman*), align=left, labelsep=-0.4em]
\item If $f_1\in \mathcal{M}^{-}(f_3)$ and $f_2\in \mathcal{M}^{+}(f_3)$, then $f_1\in \mathcal{M}^{*}(f_2)$.
\item If $f_1, f_2\in \mathcal{M}^{\pm}(f_3)$, then $f_1\in \mathcal{M}^{\pm}(f_2)$.
\end{enumerate}
\end{lemma}

\begin{proof}
Let $T_{jk}\colon N_{f_j}M\to N_{f_k}M,\;1\leq j,k\leq3,\; j\neq k,$ be orientation and mean curvature vector field-preserving,
parallel vector bundle isometries.
Denote by $Q_{jk}$ and $\Phi_j$ the distortion differential of the pair $(f_j,f_k)$ 
and the Hopf differential of $f_j$, respectively. From Lemma \ref{qke}(i), we know that $Q_{jk}$ is independent of $T_{jk}$. Hence,
$$Q_{12}=\Phi_{1}-T^{-1}_{12}\circ \Phi_{2}=\Phi_{1}-(T_{31}\circ T^{-1}_{32})\circ \Phi_{2},$$
or equivalently,
$$T^{-1}_{31}\circ Q_{12}=T^{-1}_{31}\circ \Phi_{1}-T^{-1}_{32}\circ \Phi_{2}=Q_{31}-Q_{32}.$$
Therefore,
\be \label{T12}
Q^{\pm}_{12}=T_{31}\circ (Q^{\pm}_{31}-Q^{\pm}_{32})
\ee
and the results follow immediately.
\qed
\end{proof}
\medskip

\noindent{\emph{Proof of Theorem \ref{basic}:}}
We claim that if there exist $f_1,f_2\in \mathcal{N}^{\pm}(f)$ with $f_1\in \mathcal{N}^{\pm}(f_2)$,
then the Gauss lift $G_{\pm}$ of $f$ is vertically harmonic. To unify the notation, set $f_3=f$ and denote by $Q_{jk}$ 
the distortion differential of the pair $(f_j,f_k)$.
Let $\widehat{Z}$ be the set containing the zeros of the holomorphic differentials
$Q^{\pm}_{jk}$, $1\leq j,k\leq3,\;j\neq k$. 
According to Lemma \ref{zeros}, $\widehat{Z}$ consists of isolated points only.
Appealing to Proposition \ref{angles}, we know that 
there exists $\th^{\pm}_{j}\in \mathcal{C}^{\infty}( M\smallsetminus \widehat{Z})$, $j=1,2$, with values in $(0,2\pi)$ such that
\begin{eqnarray}
Q^{\pm}_{3j}= (1-e^{\mp i\th^{\pm}_{j}})\Phi^{\pm},\label{pr}
\end{eqnarray}
where $\Phi$ is the Hopf differential of $f$. Then we have
$$Q^{\pm}_{31}-Q^{\pm}_{32}= (e^{\mp i\th^{\pm}_{2}}-e^{\mp i\th^{\pm}_{1}})\Phi^{\pm}.$$
Since the zeros of $Q^{\pm}_{12}$ are contained in $\widehat Z$, the above and \eqref{T12} imply 
that $\th^{\pm}_{1}\neq \th^{\pm}_{2}$ at every point in $M\smallsetminus \widehat{Z}$. 
Then (\ref{B}), viewed as a polynomial equation has three distinct roots, namely $e^{\pm i\th^{\pm}_{1}}, e^{\pm i\th^{\pm}_{2}}, 1$.
Therefore, $A^{\pm}=0$. From (\ref{lm}) it follows that $\th^{\pm}_{1}$ is
harmonic on $M\smallsetminus \widehat{Z}$. Since it is bounded and $\widehat{Z}$ consists of isolated points, it can be extended 
to a bounded harmonic function on $M$, which has to be constant by the maximum principle. 
Then (\ref{pr}) shows that $\Phi^{\pm}$ is holomorphic. 
Proposition \ref{glphi} implies that the Gauss lift $G_{\pm}$ is vertically harmonic and this proves the claim.

(i) Suppose to the contrary that there exist noncongruent $f_1,f_2\in \mathcal{M}^{\pm}(f)\subset \mathcal{N}^{\pm}(f)$.
From Lemma \ref{tri}(ii), we have that $f_1\in \mathcal{M}^{\pm}(f_2)\subset \mathcal{N}^{\pm}(f_2)$.
Therefore, the Gauss lift $G_{\pm}$ is vertically harmonic, a contradiction. 

For the second assertion, assume that there exists $f_1\in \mathcal{M}^{\mp}(\tilde{f})$. 
If $f_1\in \mathcal{M}^{\mp}(f)$, then Lemma \ref{tri}(ii) implies that $f\in  \mathcal{M}^{\mp}(\tilde{f})$, which is a contradiction.
Therefore, $f_1\not\in \mathcal{M}^{\mp}(f)$ and thus,
$\{\tilde{f}\}\cup\mathcal{M}^{\mp}(\tilde{f})\subset \mathcal{N}^{\pm}(f)$, which obviously holds if 
$\mathcal{M}^{\mp}(\tilde{f})=\emptyset$.
The converse inclusion is obvious if $\mathcal{N}^{\pm}(f)=\{\tilde{f}\}$.
Assume that there exists $f_1\in  \mathcal{N}^{\pm}(f)\smallsetminus \{\tilde{f}\}$. From the claim proved above,
it follows that $f_1\in \mathcal{M}^{\mp}(\tilde{f})$ and thus, 
$\mathcal{N}^{\pm}(f)\subset\{\tilde{f}\}\cup\mathcal{M}^{\mp}(\tilde{f})$. 

(ii) Suppose to the contrary that there exist noncongruent $f_1,f_2\in \mathcal{M}^*(f)$. Since both $G_{+}$ and $G_{-}$ are not vertically harmonic,
from the above claim we obtain that $f_1\not \in \mathcal{N}^+(f_2)\cup \mathcal{N}^-(f_2)$, which is a contradiction since $(f_1,f_2)$ is a Bonnet pair.

If $M$ is homeomorphic to the sphere, then for any 
$\tilde{f} \in \mathcal{M}(f)\smallsetminus \{f\}$,
the fourth-order differential $\langle Q^-,Q^+\rangle$ is holomorphic with zero-set $Z^-\cup Z^+$, 
where $Q$ is the distortion differential of the pair $(f,\tilde{f})$.
From the Riemann-Roch theorem we have that $\langle Q^-,Q^+\rangle\equiv0$. 
Hence, either $Q^-\equiv0$ or $Q^+\equiv0$ and consequently $\mathcal{M}^{*}(f)=\emptyset$. 
\qed
\medskip

\noindent{\emph{Proof of Theorem \ref{main}:}}
Theorem \ref{basic} implies that $\mathcal{M}(f)\smallsetminus \{f\}$ contains at most three congruence classes.
Assume that $M$ is homeomorphic to $\mathbb{S}^2$.
Theorem \ref{basic} shows that $\mathcal{M}^*(f)=\emptyset$ and each one of $\mathcal{M}^+(f)$ and $\mathcal{M}^-(f)$
contains at most one congruence class. Suppose that there exist $f_1\in \mathcal{M}^+(f)$ and $f_2\in \mathcal{M}^-(f)$.
From Lemma \ref{tri}(i) it follows that $f_1\in \mathcal{M}^*(f_2)$, which contradicts Theorem \ref{basic}(ii).
Therefore, $\mathcal{M}(f)\smallsetminus \{f\}$ contains at most one congruence class.
\qed
\medskip

\noindent{\emph{Proof of Theorem \ref{thr4}:}}
The proof follows immediately from Remark \ref{Eucl} and Theorem \ref{main}. \qed

\begin{remark}{\em
In the proof of Theorem $\ref{basic}$, compactness is only required for the use of the maximum principle. This theorem
and also Theorems $\ref{thr4}$, $\ref{main}$ and the results of the next subsection, still hold true if $M$ is parabolic.
In particular, this includes the case where $M$ is complete with non-negative Gaussian curvature.}
\end{remark}

\subsection{Applications to certain classes of surfaces} \label{Applications}

The following result due to Lawson-Tribuzy \cite{LT}, is an easy application of Theorem \ref{basic}.

\begin{theorem}
Let $M$ be a compact oriented 2-dimensional Riemannian manifold and
$h\in \mathcal{C}^{\infty}(M)$. If $h$ is not constant, then there exist at most two congruence classes of
isometric immersions of $M$ into $\Q^3_c$ with mean curvature $h$. In particular, there exists at most one congruence class 
if $M$ is homeomorphic to $\mathbb{S}^2$.
\end{theorem}

\begin{proof}
Suppose that there exist noncongruent isometric immersions $f_1, f_2\colon M\to \Q^3_c$ with mean curvature function $h$, and let 
$\xi_1, \xi_2$ be their unit normal vector fields. 
Consider a totally geodesic inclusion $j\colon \Q^3_c\to \Q^4_c$. The isometric immersion $\hat{f}_k=j\circ f_k\colon M\to \Q^4_c$,
$k=1,2,$ has non-parallel mean curvature vector field $hj_*\xi_k$. Proposition \ref{glphi} implies
that it has non vertically harmonic Gauss lifts.
The parallel vector bundle isometry
$T\colon N_{\hat{f}_1}M\to N_{\hat{f}_2}M$ given by 
$Tj_*\xi_1=j_*\xi_2$, $T(\Jp_{1}j_*\xi_1)=\Jp_{2}j_*\xi_2$, preserves the mean curvature vector fields,
where $\Jp_k$ is the complex structure of the normal bundle of $\hat{f}_k,\; k=1,2$.
Since the image of the second fundamental form of ${\hat{f}_k}$ is contained in the line bundle spanned by $j_{*}\xi_k, k=1,2$,
it follows that $Z=Z^-=Z^+$. Hence, $\hat{f}_2\in \mathcal{M}^*(\hat{f}_1)$
and the proof follows from Theorem \ref{basic}(ii).
\qed
\end{proof}

\begin{theorem} \label{superconformal}
Let $f\colon M\to \Q^4_c$ be a superconformal surface. If $M$ is compact and oriented,
then there exists at most one nontrivial congruence class of isometric immersions of $M$ into $\Q^4_c$,
with the same mean curvature with $f$.
\end{theorem}

\begin{proof}
Assume that $f$ is non-minimal and let $(f,\tilde{f})$ be a Bonnet pair. From Lemmas \ref{pseudo}(ii) and 
\ref{isolated} it follows that the normal curvature 
is everywhere non-vanishing on $M\smallsetminus Z$. Therefore, $\pm K_N\geq0$ on $M$.
Lemma \ref{pseudo}(ii) implies that $\Phi^{\pm}\equiv0$ and thus, $\tilde{f}\in \mathcal{M}^{\mp}(f)$.
 
We claim that $G_{\mp}$ is not vertically harmonic. Arguing indirectly, assume that $G_{\mp}$ is vertically harmonic.
Since $\Phi^{\pm}\equiv0$, Proposition \ref{glphi} yields that $G_{\pm}$ is vertically harmonic.
Then the mean curvature vector field of $f$ is parallel.
Therefore, $K_N=0$ on $M$ and Lemma \ref{pseudo}(ii) implies that $f$ is totally umbilical.
This contradicts Lemma \ref{isolated}, and the proof of the claim follows.
Hence, from Theorem \ref{basic}(i) we obtain that $\mathcal{M}^{\mp}(f)=\{\tilde{f}\}$
and consequently, $\mathcal{M}(f)=\{f,\tilde{f}\}$. 
In the case where $f$ is minimal, the result follows from \cite{Joh} or \cite{Vl1}.
\qed
\end{proof}
\medskip

We give an application to Lagrangian surfaces in $\R^4$.
Let $\widetilde J$ be a canonical complex structure on $\R^4$ which is compatible with the orientation, i.e., 
for orthonormal vectors $e_1,e_2\in \R^4$, the oriented orthonormal basis $\{e_1,e_2,\widetilde{J}e_1,\widetilde{J}e_2\}$ is 
in the orientation of $\R^4$. 
Denote by $\Omega(\cdot,\cdot)=\<\cdot,\tilde{J}\cdot\>$ the associated K\"{a}hler form.
A surface $f\colon M\to \R^4$ is called Lagrangian if $f^*\Omega=0$. 
In such a case, from
$(\widetilde{J}\circ f_*)\circ \n= \nap \circ (\widetilde{J}\circ f_*)$
we have that
$\widehat{J}_f=\widetilde{J}\circ f_*\colon TM\to N_fM$ is a parallel vector bundle isometry
and the second fundamental form of $f$ satisfies
$\a(X,Y)=\widehat{J}_fA_{\widehat{J}_fX}Y$, $X,Y\in TM$.
Thus, the trilinear map $C_f$ on $TM$ given by
$$C_f(X,Y,Z)=\Omega(\a(X,Y),f_*Z)$$
is symmetric.
Associated to $f$ are its {\emph{mean curvature form}} $\Upsilon_f$ and the cubic differential $\Theta_f$, given by
\begin{eqnarray*}
\Upsilon_f=\Omega (H,f_* \dz)dz, \;\; \Theta_f=\Omega(\add,f_*\dz)dz^3,
\end{eqnarray*}
in terms of a local complex coordinate $z$, where $\Omega$ and $\widehat{J}_f$ have been extended $\mathbb{C}$-linearly.
Since $\widetilde J$ is compatible with the orientation, $\widehat{J}_{f}\colon TM\otimes{\mathbb{C}}\to N_{f}M\otimes{\mathbb{C}}$
satisfies 
\begin{equation*}
\widehat{J}_{f}T^{(1,0)}M=N^{-}_{f}M\;\; \mbox{and}\;\; \widehat{J}_{f}T^{(0,1)}M=N^{+}_{f}M.
\end{equation*}
The {\emph {Maslov form}} $\varpi_f$ of $f$, is the $1$-form on $M$ defined by
$\varpi_f (X)=(1/\pi)\Omega(f_*X,H)$.
The Gauss map of a Lagrangian surface is
$g=(g_{+},g_{-})\colon M \to \mathbb{S}^2_{+} \times \mathbb{S}^1_{-},$
i.e., its second component lies in a great circle of $\mathbb{S}^2_{-}$.
Lagrangian surfaces with conformal (respectively, harmonic) Maslov form 
provide examples of surfaces in $\R^4$ with harmonic $g_{+}$ (respectively, $g_{-}$).
Indeed, the following was proved in \cite{CU}.
\begin{proposition}\label{Maslov}
Let $f\colon M\to (\R^4,\widetilde{J})$ be a Lagrangian surface. The following are equivalent:
\begin{enumerate}[topsep=0pt,itemsep=-1pt,partopsep=1ex,parsep=0.5ex,leftmargin=*, label=(\roman*), align=left, labelsep=0em]
\item The Maslov form $\varpi_f$ is conformal (respectively, harmonic).
\item The differential $\Theta_f$ (respectively, $\Upsilon_f$) is holomorphic.
\item The component $g_{+}$ (respectively, $g_{-}$) is harmonic.
\end{enumerate}
\end{proposition}

Using Theorem \ref{basic}, we are able to give a short proof of the following result due to He, Ma and Wang \cite{HMW}.

\begin{theorem} \label{Lagrangian}
Let $f\colon M\to (\R^4,\widetilde{J})$ be a compact, oriented Lagrangian surface
with mean curvature form $\Upsilon$.
If its Maslov form is not conformal, then there exists at most one nontrivial congruence class of 
Lagrangian isometric immersions of $M$ into $(\R^4,\widetilde{J})$, with mean curvature form $\Upsilon$.
\end{theorem}

\begin{proof}
Suppose that $f,\tilde{f}\colon M\to (\R^4,\widetilde{J})$ are noncongruent Lagrangian surfaces
with mean curvature forms $\Upsilon=\tilde{\Upsilon}$.
It follows that
$T=\widehat{J}_{\tilde{f}}\circ\widehat{J}^{-1}_{f}\colon N_{f}M\to N_{\tilde{f}}M$ is an orientation 
and mean curvature vector field-preserving, parallel vector bundle isometry.
Let $(U,z)$ be a complex chart.
From our assumption, we have that $C_{f}(\dz,\dz,\dzb)=C_{\tilde{f}}(\dz,\dz,\dzb)$.
Hence,
$$\langle \phi^-_{f}-T^{-1}\circ \phi^-_{\tilde f},\widehat{J}_{f}\dzb\rangle \equiv 0\;\; \mbox{on}\;\; U,$$
where $\phi^-_{f}$ and $\phi^-_{\tilde f}$ are given by (\ref{sf}).
Since $\phi^-_{f}-T^{-1}\circ \phi^-_{\tilde f}\in N^-_{f}U$ and $\widehat{J}_{f}\dzb\in N^+_{f}U$, it follows that 
$\phi^-_{f}-T^{-1}\circ \phi^-_{\tilde f}\equiv0$ on $U$.
Therefore, $\tilde{f}\in \mathcal{M}^+(f)$ and the proof follows from Theorem \ref{basic}(i) and Proposition \ref{Maslov}.
\qed
\end{proof}
\medskip

In \cite{CU} it was proved that if $f\colon M\to \R^4$ is a Lagrangian isometric immersion of a compact, oriented Riemannian manifold
with conformal (respectively, harmonic) Maslov form, then ${\rm genus}(M)\leq 1$ (respectively, ${\rm genus}(M)\geq 1$).
The classification of compact, oriented Lagrangian surfaces
in $\R^4$ with conformal Maslov form, was given in \cite{CU}. It turns out that there exist
Lagrangian tori in $\R^4$ with non-parallel mean curvature vector field and conformal Maslov form. 
Lagrangian surfaces with harmonic Maslov form are Hamiltonian minimal. Examples of Hamiltonian minimal Lagrangian tori
in $\R^4$, with non-parallel mean curvature vector field, were constructed in \cite{CU1} and the complete
classification was given in \cite{HR}. Furthermore, it was proved in \cite{C} that
the only compact, orientable superconformal Lagrangian surface in $\R^4$ is the Whitney sphere.
Therefore, there exist compact, oriented non-superconformal surfaces in $\R^4$, whose only one of the
components $g_{+}$, $g_{-}$ of their Gauss map is harmonic.

\section{Surfaces with a vertically harmonic Gauss lift}\label{s4}

We need some facts about absolute value type functions (cf. \cite{EGT} or \cite{ET2}).
A smooth complex function $t$ on $M$ is called of {\emph{holomorphic type}} if locally it is expressed as
$t=t_0t_1$, where $t_0$ is holomorphic and $t_1$ is smooth without zeros. A non-negative function $u$ on $M$ is called
of {\emph{absolute value type}}, if there exists a function $t$ on $M$ of holomorphic type such that $u=|t|$.
If an absolute value type function $u$ does not vanish identically, then its zeros are isolated and they have well-defined multiplicities.
Furthermore, the Laplacian $\Delta\log u$ is still defined and smooth at the zeros.
If $M$ is compact and $u$ is an absolute value type function on $M$, then
$$\int_{M}\Delta\log u=-2\pi N(u),$$
where $N(u)$ is the number of zeros of $u$, counted with multiplicities.

In the next proposition, we show that surfaces with a vertically harmonic Gauss lift satisfy Ricci-like conditions
that extend the well-known Ricci condition (cf. \cite{L}) for CMC surfaces in 3-dimensional space forms.

\begin{proposition}\label{prop}

Let $f\colon M\to \Q^4_c$ be a non-minimal surface with mean curvature vector field $H$ and vertically harmonic Gauss lift $G_{\pm}$. Then:
\begin{enumerate}[topsep=0pt,itemsep=-1pt,partopsep=1ex,parsep=0.5ex,leftmargin=*, label=(\roman*), align=left, labelsep=0em]
\item The quadratic differential $\Psi^{\pm}=\langle \Phi^{\pm}, H^{\mp}\rangle$ is holomorphic with zero-set $Z(\Psi^{\pm})=M^{\pm}_0(f)\cup \{ p\in M: H(p)=0\}.$
\item The functions $\|H\|^2$ and $\|H\|^2-(K-c) \mp K_N$ are of absolute value type with even multiplicities.
\item We have that
\begin{eqnarray}
&\Delta \log\|H\|^2=\mp 2K_N, & \label{R3}\\
&\Delta\log\big(\|H\|^2-(K-c) \mp K_N\big)=2(2K \pm K_N)\;\;\; \mbox{if}\;\;\; \Psi^{\pm}\not \equiv 0.& \label{R4}
\end{eqnarray}
\end{enumerate}
\end{proposition}

\begin{proof}
(i) The holomorphicity of $\Psi^{\pm}$ follows from Proposition \ref{glphi}.
The zeros of $\Psi^{\pm}$ are precisely the points where
$\langle \Phi^{\pm}, H\rangle=0$, which is equivalent to $\Phi^{\pm}=0$ at points where $H\neq 0$.

(ii) Let $(U,z)$ be a complex chart. From Proposition \ref{glphi}(iv) we have
$$\nap_{\dzb}H= \pm i\Jp \nap_{\dzb}H.$$ 
This is equivalently written as
$$( H^3 \pm iH^4)_{\bar z}= \mp i\w_{34}(\dzb)( H^3 \pm iH^4),$$
where $H=H^3e_3+H^4e_4$, and $\{e_3,e_4=\Jp e_3\}$ is a local orthonormal frame field of $N_fM$.
From \cite[Lemma 9.1.]{EGT} it follows that the function $H^3 \pm iH^4$ is of holomorphic type and this proves our claim for $\|H\|^2$.

From part (i), the function $\langle \phi^{\pm}, H^{\mp}\rangle$ is holomorphic, where $\phi^{\pm}$ is given by (\ref{sf}).
Moreover,
\be \label{av}
|\langle \phi^{\pm}, H^{\mp}\rangle|^2=\frac{\lambda^4 \|H\|^2}{16}\left(\|H\|^2-(K-c) \mp K_N\right),
\ee
where $\lambda$ is the conformal factor. Clearly, the function
$$t=\frac{4\langle \phi^{\pm}, H^{\mp}\rangle}{\lambda^2(H^3 \pm iH^4)}$$
can be smoothly extended to the zeros of $H$ as a holomorphic type function. 
Since $|t^2|=\|H\|^2-(K-c) \mp K_N$, this completes the proof.

(iii) Away from the zeros of $H$, we consider the local orthonormal frame field $\{e_3=H/\|H\|, e_4=\Jp e_3\}$ of the normal bundle.
Using Proposition \ref{glphi}(iv), we find that the normal connection form is given by
\begin{equation*}
\w_{34}=\pm *d\log \|H\|.
\end{equation*}
Then (\ref{R3}) follows from (\ref{normcf}) and the above.

We choose a complex chart
with coordinate $z$, away from the zeros of $\Psi^{\pm}$. From the holomorphicity of $\Psi^{\pm}$
we have that $$\Delta \log |\langle \phi^{\pm}, H^{\mp}\rangle|^2=0.$$
Equation (\ref{R4}) follows from \eqref{av}
and the fact that $\Delta \log \lambda =-K$.
\qed
\end{proof}
\medskip

\begin{proposition} \label{charact}
Let $f\colon M\to \Q^4_c$ be a compact surface with mean curvature vector field $H$ and 
vertically harmonic Gauss lift $G_{\pm}$.
\begin{enumerate}[topsep=0pt,itemsep=-1pt,partopsep=1ex,parsep=0.5ex,leftmargin=*, label=(\roman*), align=left, labelsep=0em]
\item If $f$ is non-minimal, then $$2\chi_{N}=\pm N(\|H\|^2).$$
\item If $f$ is neither minimal nor superconformal, then $$2(2\chi\pm \chi_N)=-N\left(\|H\|^2-(K-c) \mp K_N\right).$$
\end{enumerate}
\end{proposition}

\begin{proof}
The proofs of (i) and (ii) follow immediately from Proposition \ref{prop}(ii), by integrating \eqref{R3} and \eqref{R4}, respectively.
\qed
\end{proof}
\medskip

\noindent{\emph{Proof of Theorem \ref{HT}:}}
From the assumption and Proposition \ref{prop}(i) we obtain that $\Psi^{\pm}\equiv0$.
Since $f$ is non-minimal, Proposition \ref{prop}(i-ii) implies that $\Phi^{\pm}\equiv0$.
From Lemma \ref{pseudo}(ii) it follows that $f$ is superconformal with normal curvature $\pm K_N\geq0$.
Therefore, the Euler number of the normal bundle of $f$ satisfies $\pm\chi_N\geq0$, and it vanishes if and only if $K_N=0$ on $M$.
If $\chi_N=0$, then Lemma \ref{pseudo}(ii) implies that $f$ is totally umbilical.
\qed
\medskip

\noindent{\emph{Proof of Theorem \ref{AF}:}}
(i) For any $\th \in \R$ define the symmetric section $\b^{\pm}_{\th}\in \Gamma(\text{Hom}(TM\times TM, N_fM))$ by
$$\b^{\pm}_{\th}(X,Y)=J_{\th/2}^{\perp}\left(\a(J_{\mp\th/4}X,J_{\mp\th/4}Y)-\langle X,Y\rangle H\right)+\langle X,Y\rangle H,$$
where $X,Y\in TM$, $J_{\th}^{\perp}=\cos\th I +\sin\th \Jp$ and $J_{\th}= \cos\th I +\sin\th J$.
We argue that $\b^{\pm}_{\th}$
satisfies the Gauss, Codazzi and Ricci equations. Clearly, we have that
\be\label{csff}
(\b^{+}_{\th})^{(2,0)}=\Phi^-+e^{-i\th}\Phi^+,\;\;  (\b^{-}_{\th})^{(2,0)}=e^{i\th}\Phi^-+\Phi^+
\ee
and
$$(\b^{\pm}_{\th})^{(1,1)}=\a^{(1,1)}.$$
In terms of a local complex coordinate $z=x+iy$,
the Gauss equation for $f$ is written as
$$(K-c)\lambda^4/4=\left\|\a(\dz,\dzb)\right\|^2-\left\|\a(\dz,\dz)\right\|^2,$$
where $\lambda$ is the conformal factor.
Using that $\left\|\a(\dz,\dz)\right\|= \left\|\b^{\pm}_{\th}(\dz,\dz)\right\|$, we deduce that
$\b^{\pm}_{\th}$ satisfies the Gauss equation. 
Setting $e_1=\d_x/\lambda, e_2=\d_y/\lambda$ and using (\ref{csff}), we have that
$$(\b^{\pm}_{11}-\b^{\pm}_{22})\wedge \b^{\pm}_{12}=(\a_{11}-\a_{22})\wedge \a_{12},$$
where $\a_{ij}=\a(e_1,e_j)$ and $\b^{\pm}_{ij}=\b^{\pm}_{\th}(e_i,e_j),\; i,j=1,2$. Therefore, $\b^{\pm}_{\th}$ satisfies (\ref{Ricci}).
Using (\ref{CC}) and Proposition \ref{glphi}, (\ref{csff}) gives that
$$\nap_{\dzb}\b^{\pm}_{\th}(\dz,\dz)=\frac{\lambda^2}{2}\nap_{\dz}H$$
and thus, $\b^{\pm}_{\th}$ satisfies the Codazzi equation. 
By the fundamental theorem of submanifolds, for every $\th \in \R$ there exists an isometric immersion
$f^{\pm}_{\theta}\colon M\to \Q^4_c$ and an orientation-preserving 
parallel vector bundle isometry $T_{\th}\colon N_fM\to N_{f^{\pm}_{\th}}M$ such that
$\a_{f^{\pm}_{\th}}=T_{\th}\circ \b^{\pm}_{\th}$. 
Clearly, $T_{\th}H$ is the mean curvature vector field of $f^{\pm}_{\th}$,
for any $\theta \in \mathbb{S}^1\simeq \R/2\pi \mathbb{Z}$ and $f^{\pm}_0=f$. 

(ii) From Proposition \ref{HOLGL} it follows that $\Phi^{\pm}\equiv0$. Then, (\ref{csff}) yields that $(\b^{\pm}_{\th})^{(2,0)}=\Phi$ for any
$\th \in \mathbb{S}^1$. This implies that each $T_{\th}$ preserves the Hopf differential and the mean curvature vector field and 
consequently, it preserves the second fundamental form as well. This shows that the family is trivial.

(iii) Without loss of generality, we may assume that $\tilde{\th}=0$.
The distortion differential of the pair $(f,f^{\pm}_{\th})$ vanishes identically.
Lemma \ref{qke}(i) implies that any orientation and mean curvature vector field-preserving parallel 
vector bundle isometry $T\colon N_fM\to N_{f^{\pm}_{\th}}M$ preserves the Hopf differential, and consequently
the second fundamental form as well. Hence, $\b^{\pm}_{\th}=T_{\th}^{-1}\circ \a_{f_{\th}}=\a$
and (\ref{csff}) implies that $(1-e^{\mp i\th})\Phi^{\pm}\equiv0$.
Since $\th\neq0$, the last relation yields $\Phi^{\pm}\equiv0$ and thus, $f$ is superconformal.
\qed
\smallskip

\begin{proposition}\label{IAF}
Let $f\colon M\to \Q^4_c$ be a simply-connected surface with vertically harmonic Gauss lift ${G}_{\pm}$.
If $f$ is neither minimal nor superconformal, then $\{f\}\cup \mathcal{M}^{\pm}(f)=\mathbb{S}^1$.
\end{proposition}

\begin{proof}
Let $\hat f\in \mathcal{M}^{\pm}(f)$.
Appealing to Proposition \ref{angles}, we have that the distortion differential of the pair $(f,\hat{f})$ is written as
$$Q_{f,\hat{f}}=(1-e^{\mp i\th^{\pm}})\Phi^{\pm}\;\;\; \mbox{on}\;\;\; M\smallsetminus Z(Q_{f,\hat{f}}),$$
where $\th^{\pm}\in\mathcal{C}^{\infty}(M\smallsetminus Z(Q_{f,\hat{f}}))$.
From Lemma \ref{qke}(i) and Proposition \ref{glphi}, we know that $Q_{f,\hat{f}}$ and $\Phi^{\pm}$ are holomorphic.
Therefore, $\th^{\pm}$ is constant. Using (\ref{csff}), we obtain that
$Q_{f,f^{\pm}_{\th^{\pm}}}\equiv Q_{f,\hat{f}}$
and thus, $Q_{\hat{f},f^{\pm}_{\th^{\pm}}}\equiv0$.
Consequently, $\hat f$ is congruent to $f^{\pm}_{\th^{\pm}}$ and this completes the proof.
\qed
\end{proof}
\medskip

The following proposition determines the moduli space of simply-connected surfaces with parallel mean curvature vector field.
The two-parameter family given here, coincides up to a parameter transformation,
with the one given by Eschenburg-Tribuzy \cite{ET2}. 

\begin{proposition} \label{CMC}
Let $f\colon M\to \Q^4_c$ be a simply-connected surface with parallel mean curvature vector field $H\neq 0$. 
Then:
\begin{enumerate}[topsep=0pt,itemsep=-1pt,partopsep=1ex,parsep=0.5ex,leftmargin=*, label=(\roman*), align=left, labelsep=0em]
\item There exists a two-parameter family of isometric immersions
$f_{\th,\varphi}\colon M\to \Q^4_c$, $(\th,\varphi) \in \mathbb{S}^1\times\mathbb{S}^1$,
which have the same mean curvature with $f_{0,0}=f$.
\item The family is trivial if and only if $f$ is totally umbilical.
\item If $f$ is not totally umbilical, then $\mathcal{M}(f)=\mathbb{S}^1\times\mathbb{S}^1$.
\end{enumerate}
\end{proposition}

\begin{proof}
(i) Since both Gauss lifts are vertically harmonic, from Theorem \ref{AF} we may consider the two-parameter family
$f_{\th,\varphi}=(f^{-}_{\th})^{+}_{\varphi}, \th,\varphi \in \mathbb{S}^1$. Clearly, $f_{\th,\varphi}$ has the same mean curvature
with $f$.

(ii) From Theorem \ref{AF}, it is clear that $f_{\th,\varphi}$ is congruent to 
$f_{\tilde{\th},\tilde{\varphi}}$ for $(\th,\varphi)\neq (\tilde{\th},\tilde{\varphi})\in \mathbb{S}^1\times \mathbb{S}^1$ if and only if
$f$ is superconformal. Since $H$ is parallel, this can only occur if $f$ is totally umbilical.

(iii) Let $\hat{f}\in \mathcal{M}(f)$. Since $M_0=M_1\subset Z(Q_{f,\hat{f}})$, from Proposition \ref{angles} it follows that
$$Q_{f,\hat{f}}=(1-e^{i\th^-})\Phi^-+(1-e^{-i\th^+})\Phi^+\;\; \mbox{on}\;\; M\smallsetminus Z(Q_{f,\hat{f}}),$$
for $\th^+, \th^-\in \mathcal{C}^{\infty}(M\smallsetminus Z(Q_{f,\hat{f}}))$,
with $\th^{\pm}=0$ if $\hat{f}\in \mathcal{M}^{\mp}(f)$.
From Lemma \ref{qke}(i) and Proposition \ref{glphi}, we know that $Q_{f,\hat{f}},\Phi^-$ and $\Phi^+$ are holomorphic.
The above relation yields that the functions $\th^-$ and $\th^+$ are constant.
Using (\ref{csff}), one easily checks that the distortion differential $Q$ of the
pair $(f,f_{\th,\varphi})$ is given by
\be \label{ddcmc}
Q=(1-e^{i\th})\Phi^-+(1-e^{-i\varphi})\Phi^+.
\ee
From (\ref{ddcmc}) it follows that the distortion differentials of the pairs $(f,\hat{f})$ and $(f,f_{\th^-,\th^+})$
are equal and thus, the distortion differential of the pair $(\hat{f},f_{\th^-,\th^+})$ vanishes identically.
Therefore, $\hat{f}$ is congruent to $f_{\th^-,\th^+}$ and this completes the proof.
\qed
\end{proof}

\begin{remark}
{\em (i) It is clear that if $f$ is not totally umbilical, then $\{f\}\cup\mathcal{M}^-(f)=\mathbb{S}^1\times \{0\}$ and 
$\{f\}\cup\mathcal{M}^+(f)=\{0\}\times \mathbb{S}^1$.

(ii) We recall (cf. \cite{Chen, Yau}) that any surface with parallel mean curvature 
vector field $H\neq 0$ splits as 
$f=j\circ f'$, where $j\colon \Q^3_{c'}\to \Q^4_{c}, c'\geq c$, is a totally umbilical inclusion and 
$f'\colon M\to \Q^3_{c'}$ is a CMC-$h'$ surface with $h'=\pm(\|H\|^2-(c'-c))^{1/2}$.
It is known that there exists locally a bijective correspondence (the so-called Lawson correspondence \cite[Theorem 8]{L}) between CMC surfaces
in 3-dimensional space forms. Since $f_{\th,\varphi}=\hat{j}\circ\hat{f}_{\th,\varphi}$ and $\|H_{f_{\th,\varphi}}\|=\|H\|$,
the surfaces $f'$ and $\hat{f}_{\th,\varphi}$ are in Lawson correspondence for any $\th,\varphi\in \mathbb{S}^1$. In particular, 
$f_{\th,2\pi-\th}$ is congruent to $j\circ f'_{\th}$, where $f'_{\th}$, $\th\in \mathbb{S}^1$, 
is the associated family of $f'$ in $\Q^3_{c'}$ as a CMC surface.}
\end{remark}

\subsection{The moduli space of non-simply-connected surfaces}

Let $M$ be a 2-dimensional oriented Riemannian manifold with nontrivial fundamental group and
$f\colon M\to \Q^4_c$ a non-minimal isometric immersion. 
Consider the universal cover
$(\tilde{M},\tilde\pi)$ of $M$, equipped with metric and orientation that make the covering map $\tilde\pi\colon \tilde{M}\to M$ an 
orientation-preserving local isometry. Then, $\tilde{f}=f\circ\tilde\pi\colon \tilde{M}\to \Q^4_c$ is an isometric immersion.
It is clear that the Gauss lift $\tilde{G}_{\pm}$ of $\tilde{f}$ is vertically harmonic if and only if 
the Gauss lift $G_{\pm}$ of $f$ is vertically harmonic.

If $(f=f_1,f_2)$ is a Bonnet pair, then $(\tilde{f}_1, \tilde{f}_2)$ is also a Bonnet pair, where $\tilde{f}_j=f_j\circ\tilde\pi,\;j=1,2$.
Moreover, $f_2\in\mathcal{N}^{\pm}(f_1)$, if and only if
$\tilde{f}_2\in\mathcal{N}^{\pm}(\tilde{f}_1)$.
If $G_{\pm}$ is vertically harmonic and $f_2\in \mathcal{M}^{\pm}(f_1)$,
then from Proposition \ref{IAF} it follows that $\tilde{f}_2$ is congruent to some $\tilde{f}^{\pm}_{\th}$
in the associated family of $\tilde{f}_1$. Therefore $\{f\}\cup\mathcal{M}^{\pm}(f)$
can be parametrized by the set
$$\left\{\th\in \mathbb{S}^1: \mbox{there exists}\;\; f_{\th}\colon M\to \Q^4_c\;\; \mbox{such that}\;\; \tilde{f}^{\pm}_{\th}=f_{\th}\circ\tilde\pi\right\}.$$
In particular, if $H$ is parallel, then by Proposition \ref{CMC}, the moduli space $\mathcal{M}(f)$ can be parametrized by the set
$$\left\{(\th,\varphi)\in \mathbb{S}^1\times\mathbb{S}^1: \mbox{there exists}\;\; f_{\th,\varphi}\colon M\to \Q^4_c\;\; \mbox{such that}\;\; 
\tilde{f}_{\th,\varphi}=f_{\th,\varphi}\circ\tilde\pi\right\}.$$

The following is essential for the proof of Theorem \ref{verha}. 
For its proof we adopt techniques used in \cite{DV,ST,Vl2}.

\begin{proposition}\label{fins1}
Let $f\colon M\to \Q^4_c$ be a non-minimal surface with mean curvature vector field $H$.
\begin{enumerate}[topsep=0pt,itemsep=-1pt,partopsep=1ex,parsep=0.5ex,leftmargin=*, label=(\roman*), align=left, labelsep=-0.4em]
\item If the Gauss lift $G_{\pm}$ of $f$ is vertically harmonic, then $\{f\}\cup\mathcal{M}^{\pm}(f)$ is either a finite set, or the circle $\mathbb{S}^1$.
\item If $H$ is parallel, then either $\mathcal{M}(f)=\mathbb{S}^1\times \mathbb{S}^1$, or it locally decomposes as 
$\mathcal{M}(f)=\mathcal{V}_0\cup\mathcal{V}_1,$
where each $\mathcal{V}_d, d=0,1$, is either empty, or a disjoint finite union of $d$-dimensional
real-analytic varieties.
\end{enumerate}
\end{proposition}

\begin{proof}
(i) We claim that for any $\s\in \mathcal{D}$ in the group of deck transformations of the universal cover 
$\tilde\pi\colon \tilde{M}\to M$, the surfaces $\tilde{f}^{\pm}_{\th}\colon \tilde{M}\to \Q^4_c$ in the associated family of
$\tilde{f}=f\circ \tilde\pi$ and $\tilde{f}^{\pm}_{\th}\circ \s$ are congruent for any $\th \in \mathbb{S}^1$.
It is sufficient to show the existence of a parallel vector bundle isometry between the normal bundles of
$\tilde{f}^{\pm}_{\th}$ and $\tilde{f}^{\pm}_{\th}\circ \s$ that preserves the second fundamental forms.
Let $T_{\th}$ be the parallel vector bundle isometry between the normal bundles of $\tilde{f}$ and $\tilde{f}^{\pm}_{\th}$
such that 
$$\a_{\tilde{f}^{\pm}_{\th}}(X,Y)=T_{\th}\left(\tilde{J}_{\th/2}^{\perp}\big(\a_{\tilde{f}}(\tilde{J}_{\mp\th/4}X,\tilde{J}_{\mp\th/4}Y)
-\langle X,Y\rangle H_{\tilde{f}}\big) +\langle X,Y\rangle H_{\tilde{f}} \right)$$
for any $X,Y\in T\tilde{M}$, where $\tilde{J}_{\th}^{\perp}=\cos \th \tilde{I}+\sin \th \tilde{J}^{\perp}$,
$\tilde{J}_{\th}=\cos \th \tilde{I}+\sin \th \tilde{J}$ and $\tilde{J}^{\perp}$, $\tilde{J}$ stand 
for the complex structures of $N_{\tilde{f}}\tilde{M}$ and $T\tilde{M}$, respectively.
Since $\s$ is a deck transformation, we have that
$\tilde{f}\circ \s=\tilde{f}$ and thus, the normal spaces satisfy $N_{\tilde{f}}\tilde{M}(p)=N_{\tilde{f}}\tilde{M}(\s(p))$
at any $p\in \tilde{M}$. We define the vector bundle isometry
$\Sigma_{\th}\colon N_{\tilde{f}^{\pm}_{\th}}\tilde{M}\to N_{\tilde{f}^{\pm}_{\th}\circ \s}\tilde{M}$
which is given pointwise by $$\Sigma_{\th}|_{p}(\xi)= T_{\th}|_{\s (p)}\circ \left(T_{\th}|_{p}\right)^{-1}(\xi),\;\; 
\xi \in N_{\tilde{f}_{\th}}\tilde{M}(p).$$
The second fundamental forms of $\tilde{f}^{\pm}_{\th}$ and $\tilde{f}^{\pm}_{\th}\circ \s$ are related
at $p\in \tilde{M}$ by
\begin{eqnarray*}
\a_{\tilde{f}^{\pm}_{\th}\circ \s}|_{p}(X,Y)&=&\a_{\tilde{f}^{\pm}_{\th}}|_{\s(p)}(\s_{*}X,\s_{*}Y)\\
&=& T_{\th}|_{\s(p)}\left(\tilde{J}_{\th/2}^{\perp}\big(\a_{\tilde{f}}|_{\s(p)}(\tilde{J}_{\mp\th/4}\s_{*}X,\tilde{J}_{\mp\th/4}\s_{*}Y)
-\langle X,Y\rangle H_{\tilde{f}}|_{\s(p)}\big)\right.\\
& &\left.+\langle X,Y\rangle H_{\tilde{f}}|_{\s(p)} \right)\\
&=& T_{\th}|_{\s(p)}\left(\tilde{J}_{\th/2}^{\perp}\big(\a_{\tilde{f}\circ \s}|_{p}(\tilde{J}_{\mp\th/4}X,\tilde{J}_{\mp\th/4}Y)
-\langle X,Y\rangle H_{\tilde{f}\circ \s}|_{p}\big)\right.\\
& &\left.+\langle X,Y\rangle H_{\tilde{f}\circ \s}|_{p} \right)\\
&=& T_{\th}|_{\s(p)}\left(\tilde{J}_{\th/2}^{\perp}\big(\a_{\tilde{f}}|_{p}(\tilde{J}_{\mp\th/4}X,\tilde{J}_{\mp\th/4}Y)
-\langle X,Y\rangle H_{\tilde{f}}|_{p}\big)+\langle X,Y\rangle H_{\tilde{f}}|_{p} \right)\\
&=& \Sigma_{\th}|_{p}\circ \a_{\tilde{f}^{\pm}_{\th}}|_{p}(X,Y)
\end{eqnarray*}
for any $X,Y\in T\tilde{M}$ and thus, $\Sigma_{\th}$ preserves the second fundamental forms.
For any section $\xi$ of $N_{\tilde{f}^{\pm}_{\th}}\tilde{M}$ we have $\Sigma_{\th}\xi=T_{\th}(\eta \circ \s^{-1})\circ \s$, where 
$\xi=T_{\th}\eta$ for a section $\eta$ of $N_{\tilde f}\tilde{M}$. Using the fact that for any section $\delta$
of $N_{\tilde f}\tilde{M}$ and any deck transformation $\s$ we have that $\nap_X(\delta \circ \s)=\nap_{\s_{*}X}\delta \circ \s$,
we obtain
\begin{eqnarray*}
(\nap_X\Sigma_{\th})\xi&=& \nap_X\left(T_{\th}(\eta \circ \s^{-1})\circ \s\right)-T_{\th}\left(\nap_X\eta \circ \s^{-1}\right)\circ \s \\
&=& \left(\nap_{\s_{*}X}T_{\th}(\eta \circ \s^{-1})\right)\circ \s - T_{\th}\left(\nap_X\eta \circ \s^{-1}\right)\circ \s \\
&=&  T_{\th}\left(\nap_{\s_{*}X}(\eta \circ \s^{-1})-\nap_X\eta \circ \s^{-1}\right)\circ \s,
\end{eqnarray*}
where, by abuse of notation, $\nap$ stands for the normal connection of $\tilde{f}, \tilde{f}^{\pm}_{\th}$ and $\tilde{f}^{\pm}_{\th}\circ \s$.
Observe that $$\nap_{\s_{*}X}(\eta \circ \s^{-1})=\nap_X\eta \circ \s^{-1},$$
and thus $\Sigma_{\th}$ is parallel and the claim has been proved.

This allows us to define a homomorphism $S_{\th}\colon \mathcal{D}\to \text{Isom}(\Q^4_c)$ for each $\th \in [0,2\pi]$, such that
$$\tilde{f}^{\pm}_{\th}\circ \s=S_{\th}\circ \tilde{f}^{\pm}_{\th},\;\; \s \in \mathcal{D}.$$
Thus, $\th \in \{f\}\cup\mathcal{M}^{\pm}(f)$ if and only if $S_{\th}(\mathcal{D})=\{I\}$.
Assume that $\{f\}\cup\mathcal{M}^{\pm}(f)$ is infinite and let $\{\th_{m}\}$ be a sequence in
$\{f\}\cup\mathcal{M}^{\pm}(f)$ which converges to some $\th_0\in [0,2\pi]$. From $S_{\th_{m}}(\mathcal{D})=\{I\}$
for all $m\in \mathbb{N}$ we obtain that $S_{\th_0}(\mathcal{D})=\{I\}$. Let $\s \in \mathcal{D}$. By the mean value theorem applied 
to each entry $(S_{\th}(\s))_{jk}$ of the corresponding matrix, we have
\be \label{mvt}
\frac{d}{d\th}(S_{\th}(\s))_{jk}(\mathring{\th}_m)=0
\ee
for some $\mathring{\th}_m$ which lies between $\th_0$ and $\th_m$. By continuity it follows that
$$\frac{d}{d\th}(S_{\th}(\s))_{jk}(\th_0)=0.$$
Consider the sequence $\{\mathring{\th}_m\}$ that converges to $\th_0$ and observe that in view of (\ref{mvt}), 
a similar argument gives
$$\frac{d^2}{d\th^2}(S_{\th}(\s))_{jk}(\th_0)=0.$$
Repeating the argument yields $$\frac{d^n}{d\th^n}(S_{\th}(\s))_{jk}(\th_0)=0$$ for any integer $n\geq 1$. 
From the definition of the associated family, it is clear that $f^{\pm}_{\th}$ depends on the parameter $\th$ in a real-analytic way.
Since $S_{\th}(\s)$ is an analytic curve in $\text{Isom}(\Q^4_c)$, we conclude that $S_{\th}(\s)=I$ for each $\s \in \mathcal{D}$,
and thus $\{f\}\cup\mathcal{M}^{\pm}(f)=\mathbb{S}^1$.

(ii) We claim that for any $\s\in \mathcal{D}$, the surfaces $\tilde{f}_{\th,\varphi}\colon \tilde{M}\to \Q^4_c$ 
and $\tilde{f}_{\th,\varphi}\circ \s$ in $\mathcal{M}(\tilde f)$ are congruent for any $(\th,\varphi) \in \mathbb{S}^1\times \mathbb{S}^1$.
Let $T_{\th,\varphi}$ be the parallel vector bundle isometry between the normal bundles of $\tilde{f}$ and $\tilde{f}_{\th,\varphi}$
such that 
$$\a_{\tilde{f}_{\th,\varphi}}(X,Y)=T_{\th,\varphi}\left(\tilde{J}_{(\th+\varphi)/2}^{\perp}
\big(\a_{\tilde{f}}(\tilde{J}_{(\th-\varphi)/4}X,\tilde{J}_{(\th-\varphi)/4}Y)
-\langle X,Y\rangle H_{\tilde{f}}\big) +\langle X,Y\rangle H_{\tilde{f}} \right)$$
for any $X,Y\in T\tilde{M}$.
We define the vector bundle isometry
$\Sigma_{\th,\varphi}\colon N_{\tilde{f}_{\th,\varphi}}\tilde{M}\to N_{\tilde{f}_{\th,\varphi}\circ \s}\tilde{M}$
which is given pointwise by $$\Sigma_{\th,\varphi}|_{p}(\xi)= T_{\th,\varphi}|_{\s (p)}\circ 
\left(T_{\th,\varphi}|_{p}\right)^{-1}(\xi),\;\; \xi \in N_{\tilde{f}_{\th,\varphi}}\tilde{M}(p).$$
As in the proof of part (i) above, it can be shown that $\Sigma_{\th,\varphi}$ is parallel and preserves the second fundamental forms,
and the claim follows.
This allows us to define a homomorphism $S_{\th,\varphi}\colon \mathcal{D}\to \text{Isom}(\Q^4_c)$ for each $\th,\varphi \in [0,2\pi]$, such that
$$\tilde{f}_{\th,\varphi}\circ \s=S_{\th,\varphi}\circ \tilde{f}_{\th,\varphi},\;\; \s \in \mathcal{D}.$$
Clearly, $(\th,\varphi) \in \mathcal{M}(f)$ if and only if $S_{\th,\varphi}(\mathcal{D})=\{I\}$.
Since $\tilde{f}_{\th,\varphi}$ is real-analytic with respect to $(\th,\varphi)$, it follows that $\mathcal{M}(f)$ is a real-analytic set.
According to Lojacewisz's structure theorem \cite[Theorem 6.3.3.]{KP}, $\mathcal{M}(f)$ locally decomposes as 
$$\mathcal{M}(f)=\mathcal{V}_0\cup\mathcal{V}_1\cup\mathcal{V}_2,$$
where each $\mathcal{V}_d, 0\leq d\leq 2$, is either empty, or a disjoint finite union of $d$-dimensional
real-analytic subvarieties. If $\mathcal{M}(f)\neq \mathbb{S}^1\times \mathbb{S}^1$, then $\mathcal{V}_2=\emptyset$ and this completes the proof.
\qed
\end{proof}

\subsection{Surfaces in $\R^4$}
In the sequel, we deal with surfaces in $\R^4$ whose one component of the Gauss map is harmonic.
We regard the Grassmannian $Gr(2,4)$ of oriented 2-planes in $\R^4$ as a submanifold in $\Lambda^2\R^4$
via the Pl\"{u}cker embedding. 
The inner product of two simple 2-vectors in $\Lambda^2\R^4$ is given by 
$$\<\< v_1\wedge v_2,w_1\wedge w_2\>\>=\det(\<v_j,w_k\>).$$
Then, $\Lambda^2\R^4$ splits orthogonally into the eigenspaces
of the Hodge star operator $\star$, denoted by $\Lambda^2_{+}\R^4$ and $\Lambda^2_{-}\R^4$, corresponding to the eigenvalues $1$ and
$-1$, respectively. An element $a\wedge b$ of $Gr(2,4)$, where $a,b$ are orthonormal vectors in $\R^4$, decomposes as 
$$a\wedge b=(a\wedge b)_{+}+(a\wedge b)_{-},\;\; \mbox{where}\;\; (a\wedge b)_{\pm}=\frac{1}{2}\big(a\wedge b\pm \star(a\wedge b)\big).$$
Therefore, $Gr(2,4)$ can be identified with the product $\mathbb{S}^2_{+}\times \mathbb{S}^2_{-}$, where $\mathbb{S}^2_{\pm}$
is the sphere of radius $1/{\sqrt 2}$ in $\Lambda^2_{\pm}\R^4$, centered at the origin.

Let $f\colon M\to \R^4$ be a non-minimal surface, with mean
curvature vector field $H$ and Gauss map 
$g=(g_{+},g_{-})\colon M\to \mathbb{S}^2_{+}\times \mathbb{S}^2_{-}$.
In terms of a local complex coordinate $z$ away from the zeros of $H$, the components of the Gauss map are given by
\begin{eqnarray}
g_{\pm}=-\frac{i}{\lambda^2}f_*\dz\wedge f_*\dzb\mp\frac{i}{\|H\|^2}H^-\wedge H^+, \label{cgm}
\end{eqnarray}
where $\lambda$ is the conformal factor. The differential $\Psi^{\pm}$ is written as
\be \label{dpsi}
\Psi^{\pm}=\psi^{\pm}dz^2,\;\; \mbox{where}\;\; \psi^{\pm}=\langle \phi^{\pm},H^{\mp}\rangle
\ee
and $\phi^{\pm}$ is given by \eqref{sf}.
The Gauss and Weingarten formulas become respectively,
\begin{eqnarray}\label{CGauss}
\widetilde{\nabla}_{\dz}f_*\dz&=&(\log\lambda^2)_{z}f_*\dz +\frac{2\psi^-}{\|H\|^2} H^-
+\frac{2\psi^+}{\|H\|^2} H^+, \label{Gzz}\\
\widetilde{\nabla}_{\dz}f_*\dzb&=&\frac{\lambda^2}{2}(H^-+H^+),\label{Gzzb}\\
\widetilde{\nabla}_{\dz}H^{\pm}&=&-\frac{\|H\|^2}{2}f_*\dz - \frac{2\psi^{\mp}}{\lambda^2}f_*\dzb
+ \frac{2\langle \nap_{\dz}H^{\pm},H^{\mp}\rangle}{\|H\|^2}H^{\pm}, \label{Wp}
\end{eqnarray}
where $\widetilde{\nabla}$ is the induced connection on the induced bundle $f^*T\R^4$.

\begin{lemma}\label{eigenfunctions}
Let $f\colon M\to \R^4$ be a non-minimal surface. If the component
$g_{\pm}$ of the Gauss map of $f$ is harmonic, then its height functions in $\Lambda^2 _{\pm}\R^4$
are eigenfunctions of the elliptic operator $\Delta +2\left( 2\|H\|^2 -K \mp K_N\right)$, corresponding to the zero eigenvalue.
\end{lemma}

\begin{proof}
Let $z$ be a local complex coordinate away from the isolated zeros of $H$ (see Proposition \ref{prop}(ii)).
By using (\ref{Gzz})-(\ref{Wp}), equation (\ref{cgm}) yields
\begin{eqnarray}
(g_{\pm})_z=\frac{4i\psi^{\pm}}{\lambda^2 \|H\|^2}f_*\dzb\wedge H^{\pm} - if_*\dz\wedge H^{\mp}. \label{gmz}
\end{eqnarray}
Differentiating (\ref{gmz}) with respect to $\bar{z}$, we obtain that the normal component of $(g_{\pm})_{z\bar{z}}$
with respect to $\mathbb{S}^2_{\pm}$ is given by 
\begin{eqnarray*}
\left((g_{\pm})_{z\bar{z}}\right)^{\perp}=-\frac{\lambda^2}{2}\left(2\|H\|^2-K\mp K_N \right)g_{\pm}.\label{gmn}
\end{eqnarray*}
For an arbitrary vector $v_{\pm}\in\Lambda^2 _{\pm}\R^4$ we have
\begin{eqnarray*}
\Delta\langle\langle g_{\pm},v_{\pm}\rangle\rangle = \langle\langle \tau(g_{\pm})+
\frac{4}{\lambda^2}\left((g_{\pm})_{z\bar{z}}\right)^{\perp},v_{\pm}\rangle\rangle, \label{lgm}
\end{eqnarray*}
where $\tau(g_{\pm})$ is the tension field of $g_{\pm}$. 
The result follows from the above and the harmonicity of $g_{\pm}$.
\qed
\end{proof}

\begin{lemma}\label{sums}
Let $f\colon M\to\R^4$ be a surface, which is neither minimal nor superconformal.
Assume that $g_{\pm}$ is harmonic and that there exist surfaces
$f_{j}\in \mathcal{M}^{\pm}(f)$ with $\tilde{f}^{\pm}_{\th_j}=f_j\circ \tilde\pi$, and vectors
$v_{\pm}^j \in \Lambda^2_{\pm}\R^4\smallsetminus\{0\}$, $j=1,\dots,n$, such that
the Gauss maps $g^j=(g^j_{+},g^j_{-})$ of $f_j$ satisfy
\be \label{sm}
\sum_{j=1}^{n}\langle\langle g^j_{\pm}, v_{\pm}^j\rangle\rangle=0.
\ee
Then:
\begin{enumerate}[topsep=0pt,itemsep=-1pt,partopsep=1ex,parsep=0.5ex,leftmargin=*, label=(\roman*), align=left, labelsep=-0.4em]
\item The differential $\mathcal{U}^{\pm}=u^{\pm}dz$ is holomorphic, where
\be \label{up}
u^{\pm}=\sum_{j=1}^{n}\langle\langle f_{j*}\dz\wedge H^{\mp}_j, v_{\pm}^j\rangle\rangle
\ee
and $H_j$ is the mean curvature vector field of $f_j$.
\item If $\mathcal{U}^{\pm}\equiv0$, then
\be \label{esm}
\sum_{j=1}^{n}e^{i\th_j}\langle\langle g^j_{\pm}, v_{\pm}^j\rangle\rangle=0.
\ee
\end{enumerate}
\end{lemma}

\begin{proof}
From (\ref{csff}) and since by definition $\Psi^{\pm}_{f_j}=\langle \Phi^{\pm}_{f_j},H^{\mp}_j\rangle$, we have that
$$\Psi^{\pm}_{f_j}=e^{\mp i\th_j}\Psi^{\pm},\;\; j=1,\dots,n.$$
Let $(U,z)$ be a complex chart. On $U\smallsetminus Z(\Psi^{\pm})$, (\ref{gmz}) yields
\be\label{gthjz}
(g^j_{\pm})_z=e^{\mp i\th_j}\frac{4i\psi^{\pm}}{\lambda^2 \|H\|^2}f_{j*}\dzb\wedge 
H^{\pm}_j-if_{j*}\dz\wedge H^{\mp}_j,\;\;
j=1,\dots,n.
\ee
Differentiating (\ref{sm}) with respect to $z$ and using (\ref{gthjz}), we find that
\be \label{up2}
u^{\pm}=\frac{4\psi^{\pm}}{\lambda^2\|H\|^2}\sum_{j=1}^{n}e^{\mp i\th_j}\langle\langle f_{j*}\dzb\wedge H^{\pm}_j, v_{\pm}^j\rangle\rangle\;\;
\mbox{on}\;\; U\smallsetminus Z(\Psi^{\pm}).
\ee
From Proposition \ref{glphi} it follows that $H^{\pm}_j$ is an anti-holomorphic section. Hence,
\be \label{anhol}
\nap_{\dz}H^{\pm}_j=0\;\;\; \mbox{and}\;\;\;
(\|H\|^2)_{z}=2\langle \nap_{\dz}H^{\mp}_j,H^{\pm}_j\rangle,\;\; j=1,\dots,n.
\ee
Differentiating (\ref{up2}) with respect to $z$, and using (\ref{Gzzb}), (\ref{Wp}), (\ref{anhol}), (\ref{cgm}) and (\ref{up2}), we obtain that
\be \label{upz1}
u^{\pm}_{z}=u^{\pm}\left(\log\frac{\psi^{\pm}}{\lambda^2\|H\|^2}\right)_{z} +2i\psi^{\pm}
\sum_{j=1}^{n}e^{\mp i\th_j}\langle\langle g^j_{\pm}, v_{\pm}^j\rangle\rangle\;\; \mbox{on}\;\; U\smallsetminus Z(\Psi^{\pm}).
\ee
On the other hand, differentiating (\ref{up}) with respect to $z$, and using (\ref{Gzz}), (\ref{Wp}), (\ref{anhol}), (\ref{cgm}) and (\ref{up}), we find that
\be \label{upz2}
u^{\pm}_{z}=u^{\pm}\left(\log\left(\lambda^2\|H\|^2\right)\right)_{z} - 2i\psi^{\pm}
\sum_{j=1}^{n}e^{\mp i\th_j}\langle\langle g^j_{\pm}, v_{\pm}^j\rangle\rangle\;\; \mbox{on}\;\; U\smallsetminus Z(\Psi^{\pm}).
\ee

(i) From (\ref{Gzzb}), (\ref{Wp}), (\ref{anhol}) and (\ref{cgm}), we have that
\be \label{uU}
u^{\pm}_{\bar z}=\frac{\lambda^2\|H\|^2}{2i}\sum_{j=1}^{n}\langle\langle g^j_{\pm}, v_{\pm}^j\rangle\rangle\;\; \mbox{on}\;\; U
\ee
and the claim follows from (\ref{sm}).

(ii) Using (\ref{upz1}) and (\ref{upz2}), we obtain that
\be \label{sup}
\sum_{j=1}^{n}e^{\mp i\th_j}\langle\langle g^j_{\pm}, v_{\pm}^j\rangle\rangle=\frac{iu^{\pm}}{4\psi^{\pm}}
\left(\log\frac{\psi^{\pm}}{\lambda^4\|H\|^4}\right)_{z}\;\; \mbox{on}\;\; U\smallsetminus Z(\Psi^{\pm})
\ee
and (\ref{esm}) follows from (\ref{sup}).
\qed
\end{proof}

\begin{theorem}\label{M2}
Let $f\colon M\to \R^4$ be a non-superconformal isometric immersion of a compact, oriented 2-dimensional Riemannian manifold, 
with mean curvature vector field $H$ and Gauss map $g=(g_{+},g_{-})\colon M\to \mathbb{S}^2_{+}\times \mathbb{S}^2_{-}$.
\begin{enumerate}[topsep=0pt,itemsep=-1pt,partopsep=1ex,parsep=0.5ex,leftmargin=*, label=(\roman*), align=left, labelsep=-0.4em]
\item If $g_{\pm}$ is harmonic and $\mathcal{\chi}\neq \mp \mathcal{\chi}_N$, then $\mathcal{M}^{\pm}(f)$ is a finite set.
\item If $H$ is parallel and $\mathcal{\chi}\neq0$, then  $\mathcal{M}(f)$ is a finite set.
\end{enumerate}
\end{theorem}

\begin{proof}
(i) Suppose that $\mathcal{M}^{\pm}(f)$ is infinite and consider surfaces
$f_j\in \mathcal{M}^{\pm}(f)$ such that $\tilde{f}^{\pm}_{\th_j}=f_j\circ \tilde\pi$, $j=1,\dots,n$,
with $0<\th_1<\dots<\th_n<\pi$ or $\pi<\th_1<\dots<\th_n<2\pi$.
We prove that the height functions of the $\Lambda^2_{\pm}\R^4$-component 
of the Gauss maps of $f_j$ are linearly independent.
Suppose to the contrary that (\ref{sm}) holds for vectors
$v_{\pm}^j \in \Lambda^2_{\pm}\R^4\smallsetminus\{0\}$, $j=1,\dots,n$. 

We claim that $\mathcal{U}^{\pm}\equiv0$. Arguing indirectly, assume that $\mathcal{U}^{\pm}\not\equiv0$. 
From Lemmas \ref{zeros} and \ref{sums}(i), it follows that its zero-set
$Z(\mathcal{U}^{\pm})$ is isolated. 
Let $z$ be a complex coordinate in a connected
neighbourhood $U\subset M\setminus \left(Z(\Psi^{\pm})\cup Z(\mathcal{U}^{\pm})\right)$. 
From (\ref{upz1}) and (\ref{upz2}), we obtain
$$\left(\log\frac{\psi^{\pm}}{(u^{\pm})^2}\right)_{z}=0.$$
Using Proposition \ref{prop}(i) and Lemma \ref{sums}(i), we have
$$\left(\log\frac{\psi^{\pm}}{(u^{\pm})^2}\right)_{\bar z}=0.$$
Therefore, 
\be \label{c}
\psi^{\pm}=c(u^{\pm})^2 
\ee
on $U$, for a non-zero constant $c\in \mathbb{C}$.
It is easy to see that $c$ is independent of the complex coordinate and thus, $\Psi^{\pm}= c\; \mathcal{U}^{\pm}\otimes \mathcal{U}^{\pm}$ on $M$.
We argue that $Z(\Psi^{\pm})=Z(\mathcal{U}^{\pm})\neq \emptyset$.
Indeed, if $Z(\Psi^{\pm})=\emptyset$, then the holomorphic differential $\Psi^{\pm}$ is everywhere nonvanishing 
and by the Riemann-Roch theorem we obtain that $\mathcal{\chi}=0$.
On the other hand, Proposition \ref{prop}(i) implies that $H$ is everywhere nonvanishing 
and Proposition \ref{charact}(i) gives $\mathcal{\chi}_N=0$.
This contradicts our assumption.
Let $Z(\Psi^{\pm})=\{p_1,\dots,p_k\}$ and consider a complex chart $(U,z)$ around $p_r, r=1,\dots,k$, with $z(p_r)=0$.
Since $\mathcal{U}^{\pm}=u^{\pm}dz$ is holomorphic, there exists a positive integer $m_r$ such that around $p_r$ we have
\be \label{uh}
u^{\pm}=z^{m_r}\hat{u},\;\; \mbox{where}\;\; \hat{u}\;\; \mbox{is holomorphic with}\;\; \hat{u}(0)\neq0.
\ee
Hence, from (\ref{c}) we have that
$|\psi^{\pm}|^2=|z|^{4m_r}|c|^2|\hat{u}|^4,$
or equivalently, bearing in mind (\ref{av}) and \eqref{dpsi}
$$\|H\|^2(\|H\|^2-K\mp K_N)=|z|^{4m_r}u_1,\;\;\mbox{where}\;\; u_1\;\; \mbox{is smooth and positive}.$$
Proposition \ref{prop}(ii) implies that there exist non-negative integers $l_r, s_r$ such that
$$\|H\|^2=|z|^{2l_r}u_2\;\;\; \mbox{and}\;\;\; \|H\|^2-K\mp K_N=|z|^{2s_r}u_3,$$
where $u_2, u_3$ are smooth and positive.
It is clear that $s_r=2m_r-l_r$.
From (\ref{sup}), by using (\ref{c}), (\ref{uh}) and the above, on $U\smallsetminus Z(\Psi^{\pm})$ we have that
$$\sum_{j=1}^{n}e^{\mp i\th_j}\langle\langle g^j_{\pm}, v_{\pm}^j\rangle\rangle=\frac{i\lambda^2\|H\|^2}{2c(u^{\pm})^2}\left(\frac{u^{\pm}}{\lambda^2\|H\|^2}\right)_{z}=\frac{i\lambda^2 z^{l_r}{\bar z}^{l_r}u_2}{2cz^{2m_r}{\hat u}^2}\left(\frac{z^{m_r}{\hat u}}{\lambda^2z^{l_r}{\bar z}^{l_r}u_2}\right)_{z},$$
or equivalently
$$\sum_{j=1}^{n}e^{\mp i\th_j}\langle\langle g^j_{\pm}, v_{\pm}^j\rangle\rangle=
\frac{1}{z^{m_r+1}}\frac{i\lambda^2u_2}{2c{\hat u}^2}\left((m_r-l_r)\frac{\hat{u}}{\lambda^2 u_2}+z\big(\frac{\hat u}{\lambda^2 u_2}\big)_{z}\right).$$
If $m_r\neq l_r$ for some $r=1,\dots,k$, then the right-hand side of the above has a pole at $z=0$, whereas the left-hand side is bounded.
Hence, $m_r=l_r=s_r$ for any $r=1,\dots,k$. Then, Proposition \ref{charact} implies that
$\mathcal{\chi}=\mp \mathcal{\chi}_N$, which is a contradiction. Therefore, $\mathcal{U}^{\pm}\equiv0$ and this proves the claim.

According to Lemma \ref{sums}(ii), (\ref{esm}) is valid, or equivalently
$$\sum_{j=1}^{n}\cos\th_j\langle\langle g^j_{\pm}, v_{\pm}^j\rangle\rangle=0\;\;\;\; \mbox{and}\;\;\;\;
\sum_{j=1}^{n}\sin\th_j\langle\langle g^j_{\pm}, v_{\pm}^j\rangle\rangle=0.$$
Eliminating $\langle\langle g^n_{\pm}, v_{\pm}^n\rangle\rangle$, we obtain
$$\sum_{j=1}^{n-1}\langle\langle g^j_{\pm}, w_{\pm}^j\rangle\rangle=0,$$
where $w_{\pm}^j=\sin(\th_n-\th_j)v_{\pm}^j\neq 0$, $j=1,\dots,n-1$.
By induction, we finally find that $\langle\langle g^n_{\pm}, w_{\pm}\rangle\rangle=0$ 
for some non-zero vector $w_{\pm}\in \Lambda^2 _{\pm}\R^4$.
Therefore, $g^n_{\pm}$ takes values in a great circle of $\mathbb{S}^2_{\pm}$ and thus,
its Jacobian $\mathcal{J}_{g^n_{\pm}}$ vanishes. 
On the other hand, we know that (cf. \cite[Proposition 4.5.]{HO})
\begin{equation*} \label{Jacobians}
K=\mathcal{J}_{g^n_{+}}+\mathcal{J}_{g^n_{-}}\;\;\;\mbox{and}\;\;\; K_N=\mathcal{J}_{g^n_{+}}-\mathcal{J}_{g^n_{-}}.
\end{equation*}
Hence, we conclude that $K= \mp K_N$, which contradicts our topological assumption.
Therefore, we have proved that the height functions of the $\Lambda^2 _{\pm}\R^4$-component 
of the Gauss maps of $f_j$ are linearly independent.
This contradicts Lemma \ref{eigenfunctions}, since the eigenspaces of an elliptic operator are finite dimensional.
Hence, $\mathcal{M}^{\pm}(f)$ is a finite set.

(ii) Assume that $\mathcal{M}(f)$ is infinite. Then there exists a sequence $f_k\in \mathcal{M}(f)$ 
such that $\tilde{f}_{\th_k,\varphi_k}=f_k\circ \tilde\pi$, for which $(\th_l,\varphi_l)\neq(\th_m,\varphi_m)$ for $l\neq m$.
Without loss of generality, we may assume that either
$0<\th_l<\th_m<\pi$, or $\pi<\th_l<\th_m<2\pi$, for $l,m\in \mathbb{N}$ with $l<m$. 
We prove that the height functions of the $\Lambda^2 _{-}\R^4$-component
of the Gauss maps of $f_j$, $j=1,\dots,n$, are linearly independent.
Suppose to the contrary that (\ref{sm}) holds for vectors $v_{-}^j \in \Lambda^2_{-}\R^4\smallsetminus \{0\}$, $j=1,\dots,n$.
From (\ref{ddcmc}) it follows that
$\Psi^{-}_{f_j}=e^{i\th_j}\Psi^{-}.$
Consequently, the relations (\ref{gthjz})-(\ref{sup}) are valid and thus, the conclusion of Lemma \ref{sums} also holds. 
Taking into account that $K_N=0$, we can prove as in the proof of part (i)
that our topological assumption implies $\mathcal{U}^{-}\equiv0$.
The remaining of the proof is the same with the one of part (i).
\qed
\end{proof}
\bigskip

\noindent{\emph{Proof of Theorem \ref{verha}:}}
(i) It follows immediately from Theorem \ref{superconformal}.

(ii) From Proposition \ref{fins1}(i) we know that $\{f\}\cup\mathcal{M}^{\pm}(f)$ is either finite, or the circle $\mathbb{S}^1$.
We show that the same holds true for the set $\mathcal{M}^*(f)\cup\mathcal{M}^{\mp}(f)$.

Suppose that $\mathcal{M}^*(f)\cup\mathcal{M}^{\mp}(f)$ is infinite.
Since $G_{\mp}$ is not vertically harmonic, from Theorem \ref{basic}(i) it follows that
$\mathcal{M}^{\mp}(f)$ contains at most one congruence class and thus,
$\mathcal{M}^*(f)$ is infinite. For $\tilde{f}\in \mathcal{M}^*(f)$, 
Theorem \ref{basic}(i) implies that
$\mathcal{M}^*(f)\cup\mathcal{M}^{\mp}(f)=\{\tilde{f}\}\cup\mathcal{M}^{\pm}(\tilde{f})$
and the proof follows from Proposition \ref{fins1}(i) applied to the surface $\tilde{f}$.

(iii) By virtue of part (i), we assume that $f$ is non-superconformal.
The case where $H$ is parallel has been proved in Theorem \ref{M2}(ii). Assume that $H$ is non-parallel and
suppose to the contrary that $\mathcal{M}(f)$ is infinite.
From Theorem \ref{M2}(i) it follows that $\mathcal{M}^{\pm}(f)$ is finite. Since $g_{\mp}$ is 
not harmonic, Theorem \ref{basic}(i) implies that $\mathcal{M}^{\mp}(f)$ contains at most one congruence class and therefore,
$\mathcal{M}^*(f)$ is infinite. Theorem \ref{basic}(i) yields that
$\mathcal{M}^*(f)\cup\mathcal{M}^{\mp}(f)=\{\tilde{f}\}\cup\mathcal{M}^{\pm}(\tilde{f})$ for any $\tilde{f}\in\mathcal{M}^*(f)$,
which contradicts Theorem \ref{M2}(i) for $\tilde{f}$.
\qed

\vspace{.1in} 
{\renewcommand{\baselinestretch}{1}
\hspace*{-20ex}\begin{tabbing} \indent\= Univ. of Ioannina -- Math. Dept.
\indent\indent\= Univ. of Ioannina -- Math. Dept. \\
\> 45110 Ioannina -- Greece  \>
45110 Ioannina -- Greece \\
\> E-mail: kpolymer@cc.uoi.gr \> E-mail: tvlachos@uoi.gr
\end{tabbing}}
\end{document}